\documentclass[11pt]{article}
\usepackage{amsmath,amssymb,amsfonts,amsthm}
\usepackage{array}
\usepackage{tikz-cd}
\usepackage[margin=1in]{geometry}
\usepackage{bm}
\usepackage{xcolor}
\usepackage[authoryear,round]{natbib}
\usepackage{tikz-3dplot}
\usetikzlibrary{arrows.meta,decorations.markings}
\usepackage{graphicx}
\usepackage[hidelinks]{hyperref}
\usepackage{cleveref}
\newif\ifshowchanges
\showchangesfalse
\ifshowchanges
  \newcommand{\chg}[1]{{\color{orange}#1}}
\else
  \newcommand{\chg}[1]{#1}
\fi

\usepackage{subcaption}

\newcommand{\safeincludegraphics}[2][]{%
  \IfFileExists{#2}{\includegraphics[#1]{#2}}{%
    \fbox{\parbox[c][1.35in][c]{0.88\linewidth}{\centering\small Figure file \texttt{\detokenize{#2}}\\not supplied with this standalone source.}}}}

\newcommand{\SO}[1]{\ensuremath{\mathrm{SO}(#1)}}

\newcommand{\Ad}{\mathrm{Ad}}

\newcommand{\R}{\mathbb{R}}
\newcommand{\Sph}{\mathrm{S}}

\theoremstyle{plain}
\newtheorem{proposition}{Proposition}
\newtheorem{theorem}{Theorem}

\theoremstyle{definition}
\newtheorem{definition}{Definition}
\theoremstyle{remark}
\newtheorem{remark}{Remark}

\title{Rotational Trajectory Synopsis Groups}
\author{
	Gregory S. Chirikjian\footnote{Division of Computing and Mathematical Sciences, Mohamed bin Zayed University of Artificial Intelligence, Abu Dhabi, UAE \texttt{gregory.chirikjian@mbzuai.ac.ae}, and Mechanical Engineering, University of Delaware, Newark, DE 19716 \texttt{gchirik@udel.edu}}
\,\,\,\,\,\,
Taeyoung Lee\footnote{Mechanical and Aerospace Engineering, The George Washington University, Washington DC, 20052 \texttt{tylee@gwu.edu}}}
\date{October 1, 2026}

\begin{document}

\maketitle

\begin{abstract}
 We construct finite-dimensional Lie-group summaries of based rotational paths that retain the exact terminal attitude and selected iterated angular-velocity integrals. These summaries, called rotational trajectory synopses, compose exactly under path concatenation. Building on the classical works of Magnus and Chen, we establish endpoint-retaining body, spatial, and hybrid synopsis groups at every finite order and give a uniform matrix realization for the evaluated Magnus hierarchy. We distinguish this hierarchy from the full Chen signature from order three onward and prove \chg{that the exact endpoint is independent of every finite truncation of the body, spatial, or combined angular-velocity signature: the synopsis maps are surjective at every finite order}. An underactuated satellite example connects the retained coordinates to geometric phase in spacecraft attitude dynamics.
\end{abstract}

\section{Introduction}

Traditionally, rotations and orientations are both described using special orthogonal matrices or unit quaternions. These mathematical objects are useful to capture an orientational state at any fixed time, which can be considered as the result of a time-varying rotational motion up to that time. The actual rotational motion that may have resulted in that state could have been direct, or taken large excursions. For example, a gymnast, ballerina, diver, or figure skater can spin or tumble many times before ending up at a specific orientation. The final rotation matrix does not contain any information about how the final orientation was achieved.
This can be quite important in space applications such as when an astronaut is tethered to a spacecraft, as shown in Figure \ref{tangledastronaut}. This image is only for motivation: rotational history alone does not determine a tether's spatial embedding or knot type, but it is a start.

\begin{figure}[t]
    \centering
    \begin{subfigure}[t]{0.32\linewidth}
        \centering
        \safeincludegraphics[width=\linewidth]{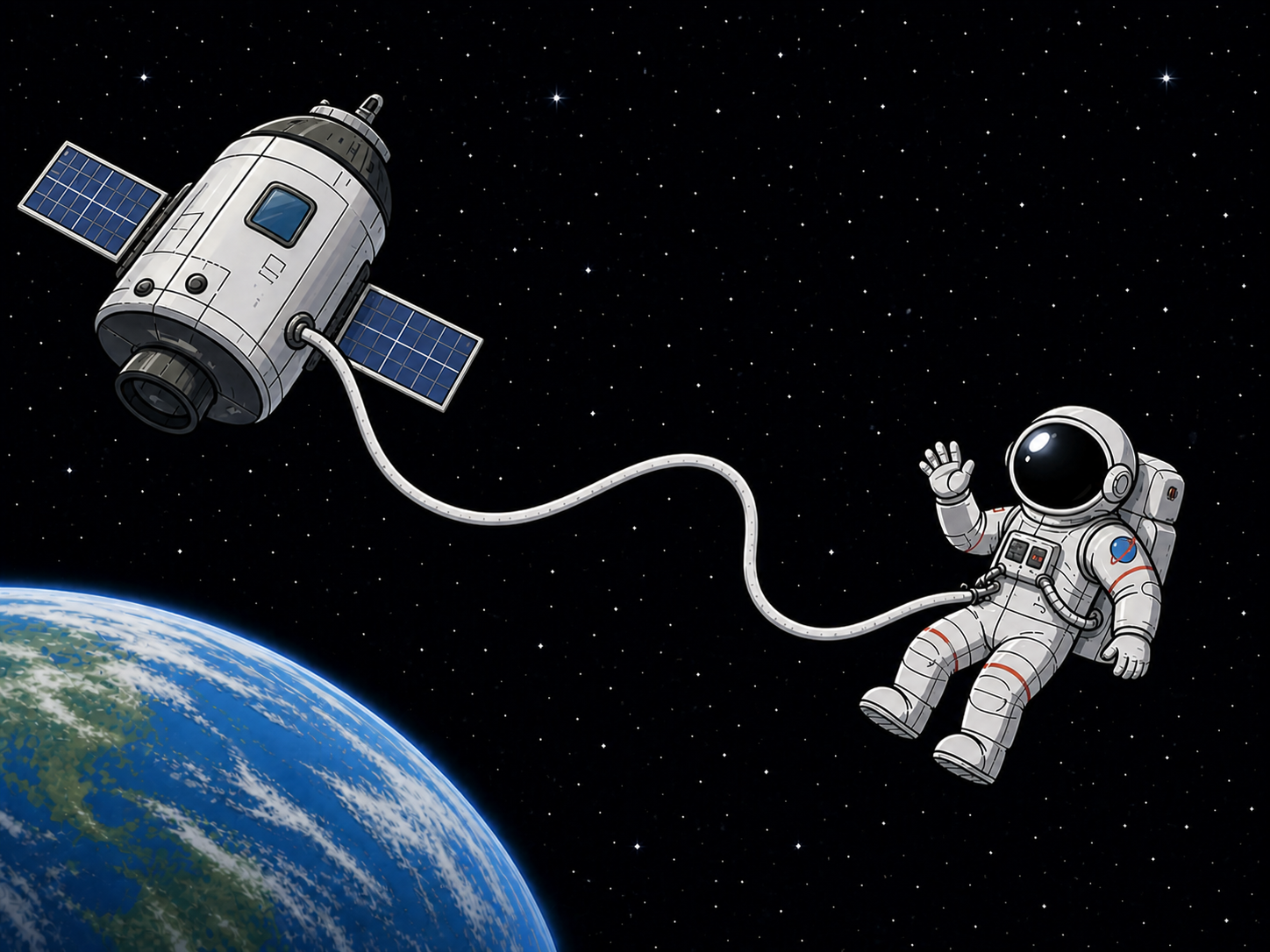}
        \caption{}
    \end{subfigure}
    \hfill
    \begin{subfigure}[t]{0.32\linewidth}
        \centering
        \safeincludegraphics[width=\linewidth]{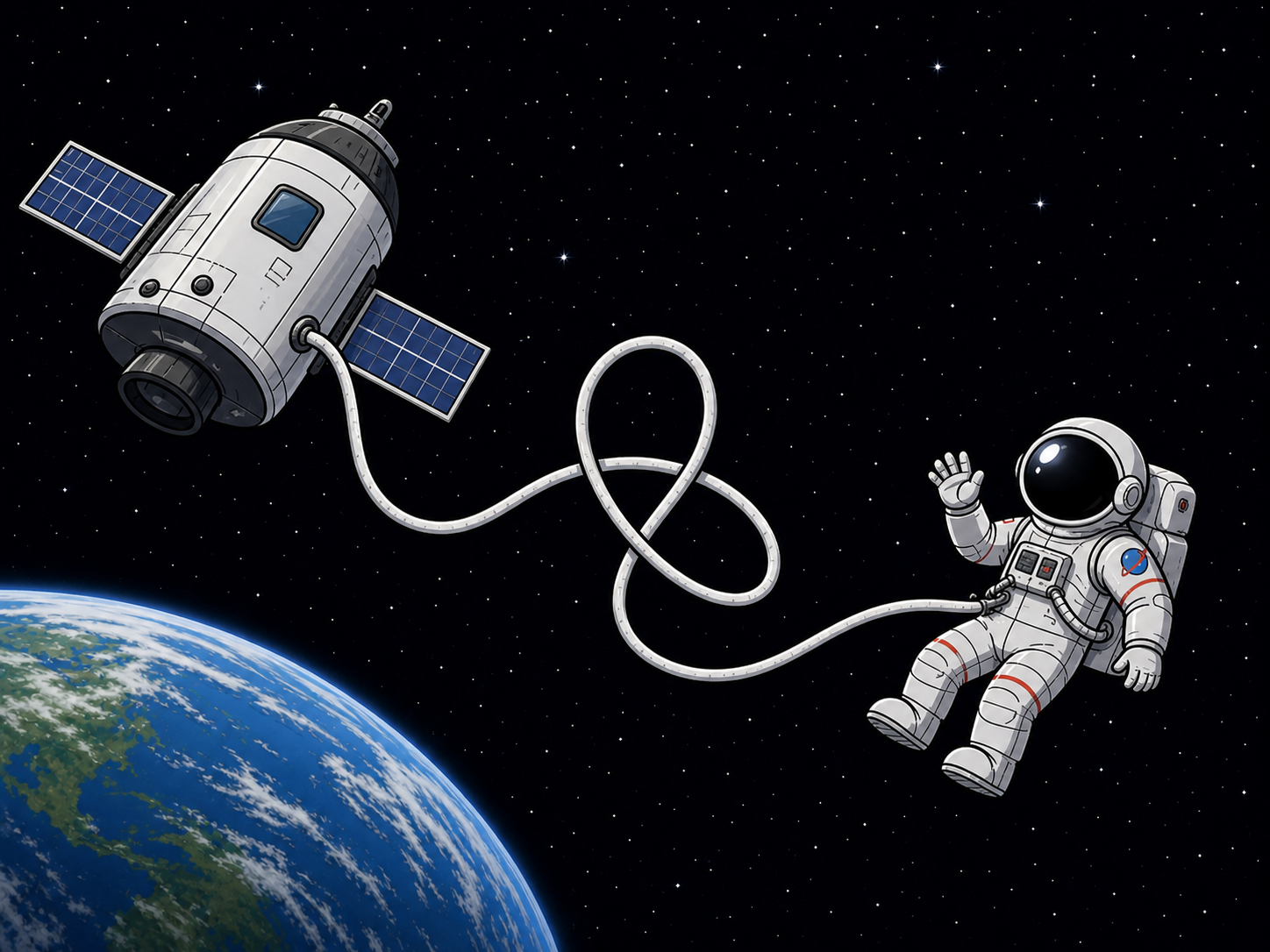}
        \caption{}
    \end{subfigure}
    \hfill
    \begin{subfigure}[t]{0.32\linewidth}
        \centering
        \safeincludegraphics[width=\linewidth]{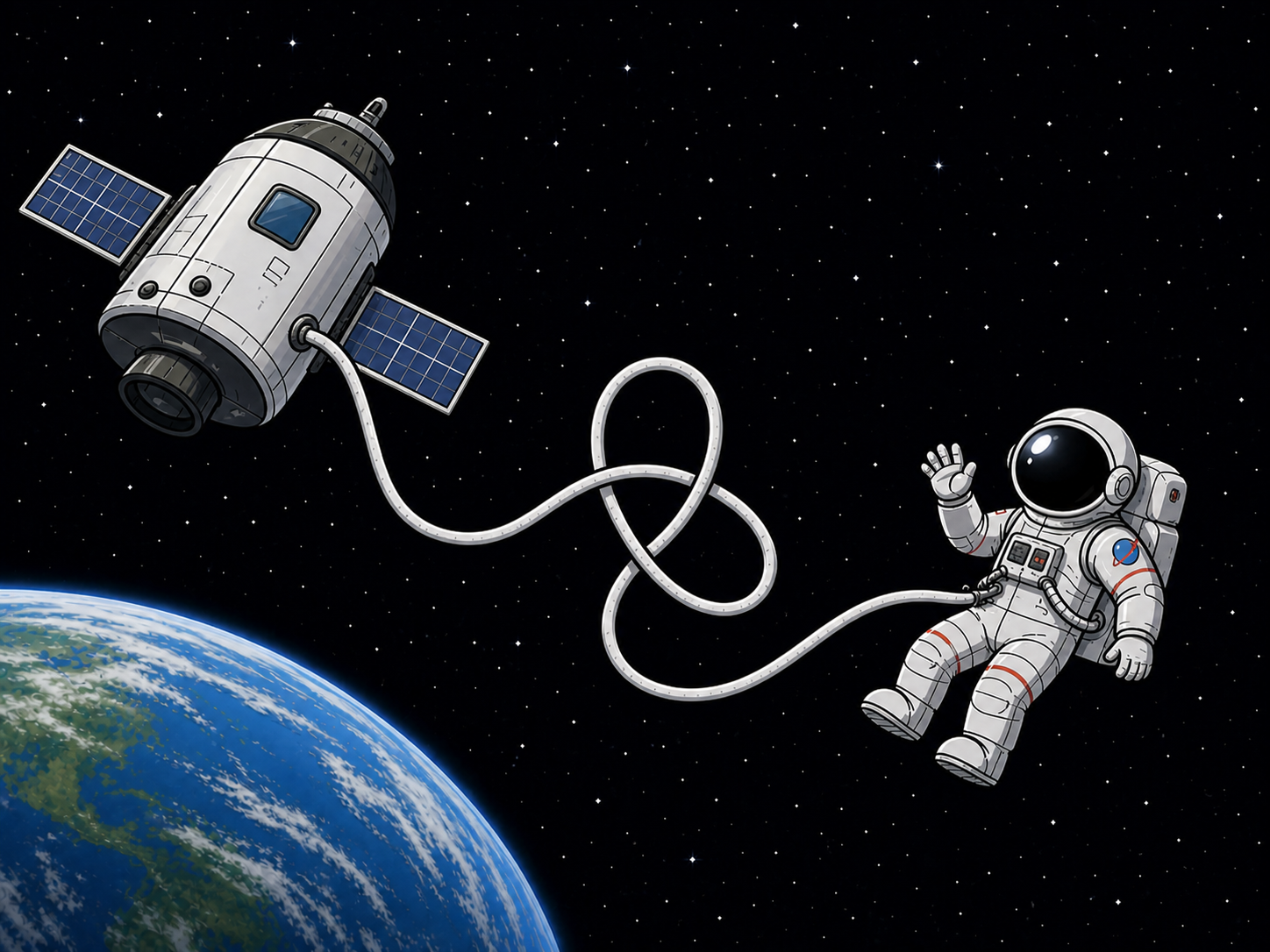}
        \caption{}
    \end{subfigure}
\caption{Illustrative tether configurations with the same terminal orientation but different spatial histories: (a) a tether without tangles; (b) a tangled tether with the loop on the left passing over; (c) a tether with loop on the left passing under that can potentially tighten. The drawings motivate the distinction between terminal orientation and path history; they do not specify a knot classification.}
\label{tangledastronaut}
\end{figure}

The whole time-parameterized rotational trajectory retains the rotational history, but the space of rotational trajectories is infinite dimensional whereas the space of fixed orientations is finite dimensional. In many applications, such as in state estimation and filtering, it is much more convenient to describe probabilities on finite-dimensional spaces than on infinite-dimensional ones. Examples addressing attitude and pose include \citet{Barfoot2024} and \citet{Bonnabel2015,Bonnabel2018}; \citet{Chirikjian2012} and \citet{Chirikjian2024} address arbitrary finite-dimensional unimodular Lie groups.

At the core of this presentation is an intermediate description which records more information about a rotational path than only the final orientation. This is done with quantities such as the integral of angular velocity which give summaries of path properties. Keeping track of these retains the desired feature of finite dimensionality, while being more descriptive of the rotational trajectory resulting in an orientation than simply using the terminal rotation matrix or quaternion.

The contribution is a constructive family of rotational descriptors that stores the exact terminal attitude together with selected accumulated body-fixed and space-fixed motion coordinates. These quantities have explicit composition laws, allowing summaries of consecutive maneuvers to be combined without retaining their sampled histories. The exact rotation supplies information that no fixed finite truncation of the angular-velocity signature determines\chg{: Proposition~\ref{prop:endpoint-adds} proves that, at every order, the synopsis maps onto the body, spatial, and combined groups are surjective, so the endpoint is independent of the truncated signatures}.
 Our contribution is an explicit hierarchy of endpoint-retaining body, spatial, and hybrid rotational synopsis groups, together with concrete matrix realizations and a comparison of the Chen and evaluated Magnus hierarchies.
The first-order space-fixed pair $(R,x_s)$, with $x_s$ denoting the time integral of space-fixed angular velocity,  is itself an element of the special Euclidean (rigid-motion) group $\mathrm{SE}(3)$ (Section~\ref{hybridsec}), and the second-order coordinate $b$ introduced later is the classical \emph{coning} integral of strapdown inertial navigation~\citep{Bortz1971,Savage1998}. 

The conventional wisdom is that rotations in $n$-dimensional Euclidean space can be described by special orthogonal matrices, the set of which is
$$ \SO{n} := \{R \in \mathbb{R}^{n\times n} \,|\, RR^T = \mathbb{I} \,\,\,{\rm and}\,\, \det R = +1\} \,.$$
These are compact connected Lie groups under the operation of multiplication of the matrices that constitute their respective elements.

It is well known that the matrix exponential of a skew-symmetric matrix is a rotation matrix, and in the three-dimensional case the dual vector associated with such matrices is used extensively.
The space of $3\times 3$ real skew-symmetric matrices is denoted by $\mathfrak{so}(3)$, and we write
$\exp: \mathfrak{so}(3) \rightarrow \SO3$.
Given a vector ${x} = [x_1,x_2,x_3]^T \in \mathbb{R}^3$, it is always possible to construct a skew-symmetric matrix
$$x^{\wedge} := \left(\begin{array}{ccc}
0 & -x_3 & x_2 \\
x_3 & 0 & -x_1 \\
-x_2 & x_1 & 0 \end{array}\right) $$
such that
$$x^{\wedge} {y} = {x} \times {y}$$
for any ${y} \in \mathbb{R}^3$ where
$\times$ is the cross product. That is, there is a bijective map
$$\wedge: \mathbb{R}^3 \rightarrow \mathfrak{so}(3)$$
and the inverse is
$$\vee: \mathfrak{so}(3) \rightarrow \mathbb{R}^3$$
such that
$$x = (x^{\wedge})^{\vee} \,.$$
This is all well known in the robotics and geometric controls and state estimation literature.
Perhaps less well known is that given any $x\neq 0$ if ${u} := {x}/\|{x}\|$ then for any
$k \in \mathbb{Z}$
$$ \exp(x^{\wedge} + 2\pi k {u}^{\wedge}) \,=\, \exp x^{\wedge} \,.$$
 We use a minimum-norm real skew-symmetric logarithm, whose rotation vector lies in a ball of radius $\pi$:
$$ y^{\wedge} := \log(\exp {x}^{\wedge})  \,\,\,{\rm where}\,\,\, \|y\|\leq \pi\,.$$
 For rotations of angle less than $\pi$, this logarithm is unique and its rotation vector has norm less than $\pi$. At angle $\pi$, the two minimum-norm rotation vectors differ by sign; a selection must be specified if the logarithm is used there. The unoriented rotation axis is determined, but no global continuous choice of its orientation exists.

\subsection*{Related work and scope}
A path signature is the sequence of its iterated integrals; its tensor logarithm is the log signature (Section~\ref{chenrev}). The group-like, or shuffle, relations are the algebraic constraints satisfied by these integrals. Informally, tree-like equivalence identifies paths differing by excursions that can be collapsed through a tree; the cited uniqueness theorem gives the precise definition.
Magnus expansions reorganize time-ordered exponentials into Lie polynomials, with convergence requiring separate hypotheses~\citep{Magnus1954}. Chen's iterated integrals and their group-like/shuffle structure underlie path signatures~\citep{Chen1957}; bounded-variation paths are determined by their full signatures only up to tree-like equivalence~\citep{HamblyLyons2010}. Signatures have also been developed directly for Lie-group-valued paths, including rotation groups~\citep{LeeGhrist2020}. (See also \citealp{Blanes2009,Iserles2000,Chen1977,Reutenauer1993,Fliess1981,Sussmann1986,Beauchard2023,MuntheKaas1998} for further background on Magnus/Chen theory and its control-theoretic antecedents.) 

A number of recent works address similar issues as in this paper. Celledoni, Lystad, and Tapia use signatures and log signatures of Lie-group paths for motion identification, including data in products of $\SO3$~\citep{Celledoni2019}. Lou, Li, and Ni study finite-dimensional Lie-group developments of paths and their multiplicativity~\citep{Lou2024}. Lee and Ghrist explicitly establish Chen's identity and the homomorphism structure for signatures on Lie groups~\citep{LeeGhrist2020}. Accordingly, concatenation-preserving group-valued path summaries have substantial precedent. In the present work, the descriptor explicitly retains the physical endpoint together with  a finite truncation of the angular-velocity signature. It combines body and spatial integrals when desired, and constructs associated concrete matrix Lie groups. That is, the formulation presented here is a specialization and extension of the established machinery, resulting in a concrete hierarchy of finite-dimensional matrix Lie groups appropriate for engineers to do computations.

Three further bodies of work are closely related: (i) In strapdown inertial navigation the second-order Magnus coordinate $b=\tfrac12\int_{0<t_1<t_2<T}\omega_b(t_1)\times\omega_b(t_2)\,dt_1dt_2$ is the classical coning term of the Bortz rotation-vector equation~\citep{Bortz1971,Savage1998}, so the geometric-phase reading of $b$ in Section~\ref{sec:geophase} has a long history there. (ii)  On-manifold inertial measurement unit (IMU) preintegration~\citep{Forster2017} also stores composable finite-dimensional increments together with an exact rotation increment; the first-order space-fixed group $\SO3\ltimes\R^3$ used here has a similar semidirect structure. (iii) Nilpotent approximation and motion planning for driftless systems~\citep{LafferriereSussmann1991} rely on truncated free nilpotent groups closely related to the $G_N$ of Section~\ref{sec:allorders}.  Proposition~\ref{prop:endpoint-adds} uses Lie-algebra generation and controllability to establish surjectivity of the body, spatial, and hybrid synopsis maps at every finite order. In particular, the exact endpoint is not determined even by the combined finite body and spatial signatures. 

\begin{table}[t]
\centering
\small
\begin{tabular}{>{$}l<{$} p{0.68\linewidth}}
\hline
\text{Symbol} & Meaning\\ \hline
\mathcal P_e(G),\ \gamma & Duration-tagged based path monoid; one path\\
\Sigma,\ \Pi,\ \sigma & Synopsis group, synopsis map, synopsis element\\
R(t) & Physical rotation in $\SO3$\\
\omega_b,\omega_s;\ \Omega_b,\Omega_s & Angular-velocity vectors; corresponding $\mathfrak{so}(3)$ matrices\\
x_b,x_s & First integrated angular-velocity vectors\\
\Xi_{n,b},\Xi_{n,s} & Homogeneous evaluated Magnus terms\\
S^{(N)} & Chen signature truncated at degree $N$\\
A;\ B=b^\wedge & Chronological second tensor; step-2 log-signature coordinate\\
Q_i,a_i & Test-frame rotation and its retained moment\\
\Gamma;\ \mathcal{F} & Reduced direction on $S^2$; oriented spanning surface\\
\kappa & Sign of the geometric-phase steering loop\\
{N} & Truncation order (\chg{Sections~\ref{chenrev}--\ref{sec:allorders}}) \\
P &  Number of modal coefficients (\chg{Section~\ref{sec:modal}}) \\
{G_N(V),\ \mathfrak m_N,\ M_N} & Free step-$N$ signature group; evaluated Magnus algebra; evaluated Magnus group \\
{\odot,\ \circ_N} & Truncated tensor product;  Baker--Campbell--Hausdorff (BCH) product on $M_N$ \\
{\varepsilon} & Formal degree label with $\varepsilon^{N+1} :=0$ \\
\chg{\iota} & \chg{Zero-duration identity path (Sections~\ref{planarsec}--\ref{pathsec})} \\
\chg{\delta} & \chg{Smallness parameter of the local loop expansion (Section~\ref{sec:magnus-geophase})} \\
{x_b,\,b_b,\,c_b,\,d_b} & $\Xi_{k,b}^\vee$ for $k=1,\dots,4$ (subscript $b$ = body; not new vectors) \\ \hline
\end{tabular}
\caption{Notation used throughout the manuscript. Lowercase angular velocities are vectors; uppercase angular velocities are skew matrices.}
\label{tab:notation}
\end{table}

The remainder of this paper is organized as follows.
Section \ref{planarsec} begins with the planar case, where the basic idea of  recording net signed planar rotation can be stated without the complications of noncommutativity.
Section \ref{pathsec} then formulates rotational trajectories as based paths with a natural concatenation operation, emphasizing the distinction between endpoint orientation and path history.
Section \ref{hybridsec} introduces group extensions based on integrated body-fixed and space-fixed angular velocities.
Sections \ref{magnussec} and \ref{chenrev} then develop higher-order synopsis variables using the Magnus expansion and Chen iterated integrals, including their matrix group realizations and their body-fixed and space-fixed versions.
Section~\ref{sec:allorders} proves the all-orders group constructions and explains the third-order Chen--Magnus distinction.
The following sections examine when these quantities can reconstruct trajectory information, illustrate the role of geometric phase in an underactuated satellite example, and conclude by summarizing the resulting hierarchy of rotational trajectory synopsis groups.

\section{Planar Rotational Synopsis Groups} \label{planarsec}

The first step in developing a synopsis of rotational paths is establishing in what space such paths evolve.
 In the planar case, $\SO2$ consists of matrices of the form
$$ R(\theta) = \left(\begin{array}{cc}
\cos \theta & - \sin\theta \\
\sin\theta & \cos \theta \end{array}\right) \,.$$

This might appear to be the end of the story, but the periodicity of the trigonometric functions destroys any path history involving multiple revolutions. The universal covering group of $SO(2)$ can be identified with the real line with $x \in \mathbb{R}$ with operation of addition, and this keeps that information. Describing ``total orientation'' as $x \in \mathbb{R}$ has more information than $\theta \in [0,2\pi)$, and is backward compatible with it via modular arithmetic. In group-theoretic language, the integers (and scaled versions thereof) form a normal subgroup of the real numbers, and the quotient group of interest is
$$ \SO2 \,\cong\, \frac{\mathbb{R}}{2\pi \mathbb{Z}} \,. $$

Writing $x = \theta + 2\pi k$ means that the 
integer counter $k \in \mathbb{Z}$ has a topological interpretation.
	The fundamental group $\pi_1(M)$ records homotopy classes of loops in a manifold $M$~\citep{hatcher2002algebraic}, and for the planar rotation group
$$
\pi_1(\SO2) \cong \mathbb{Z},
$$
so loops are classified by an integer winding number. Equivalently, the universal cover of $\SO2$ is
$$
\mathbb{R} \rightarrow \SO2, \qquad x \mapsto R(x),
$$
and the integer $k$ records the covering sheet of the lifted endpoint of a planar rotational trajectory.

The group $(\mathbb{R},+)$
measures more than the terminal rotation matrix in a planar rotational trajectory and
can be used at one level as a synopsis of a continuous rotational motion.
In this covering group the composition of two based trajectories (i.e., $x_1(t)$ and $x_2(t)$ with $x_1(0)=x_2(0)=0$) is:
\[
(x_1 \star x_2)(t) \,:=\,
\begin{cases}
x_1(t), & 0 \le t \le T_1, \\[6pt]
x_1(T_1) + x_2(t - T_1), & T_1 \le t \le T_1 + T_2 .
\end{cases}
\]
The operation $\star$ is not commutative:
$$ (x_1 \star x_2)(t)  \neq (x_2 \star x_1)(t) $$
in general, even though \SO2 and $(\mathbb{R},+)$ are commutative and both concatenations have the same endpoints. Their time histories, and even their images, can differ. A change of time parameter need not make them equivalent.

Moreover, the set of all such based trajectories, together with the concatenation operation $\star$, which is associative
$$ ((x_1 \star x_2) \star x_3)(t) = (x_1 \star (x_2 \star x_3))(t) \,,$$
is not a group. Rather, it is a monoid if zero-duration paths are admitted, and otherwise only a semigroup, since strictly speaking a trajectory does not have an inverse that results in the identity path. When zero duration is admitted, the constant path $\chg{\iota}(t) := e$ on $[0,0]$ is the identity. Concatenating a nonconstant path with a reversal still produces a positive-duration loop, not that identity.

The simplest synopsis of a based rotational trajectory over the time interval $[0,T]$ is then simply
\begin{equation}
x(T) = \int_{0}^{T} \dot{x}(t) \,dt = \theta(T) + 2\pi k(T)
\label{intxeq}
\end{equation}
for some $k \in \mathbb{Z}$.

This simple equation can be interpreted in a very meaningful way.
Away from branch-cut crossings, $\dot{x}=\dot{\theta}$. The wrapped angle $\theta$ jumps at those crossings, whereas its continuous lift $x$ is obtained by integrating the physical angular velocity. This integral keeps track of  the net accumulated signed planar rotation.
A forward rotation followed by an equal backward rotation therefore gives zero net rotation, even though the path has positive length.
Consequently, $x(T)$ gives the desired synopsis in the planar case, whereas $\theta(T)$ does not. Equivalently, this could be written as $(\theta(T), k(T))$, or as $(\theta(T), x(T))$. Though the last pair is redundant in the planar case, it
has several useful extensions in the three-dimensional case as will be seen in subsequent sections.

It is precisely a 3D analogue of this construction that is sought below. Let $\mathcal{P}_e(G)$ be the monoid of duration-tagged, piecewise-$C^1$ based paths $(T,\gamma)$, with $T\geq0$, $\gamma:[0,T]\to G$, $\gamma(0)=e$,  raw concatenation $\star$ as defined in Section~\ref{sec:path-composition}, and the unique zero-duration path $\chg{\iota}$ as identity.

\begin{definition}[Finite-dimensional group synopsis]
A synopsis is a pair $(\Sigma,\Pi)$ consisting of a finite-dimensional Lie group $\Sigma$ and a monoid homomorphism $\Pi:\mathcal{P}_e(G)\to \Sigma$:
\[
\Pi(\chg{\iota})=e_{\Sigma},\qquad
\Pi(\gamma_1\star\gamma_2)=\Pi(\gamma_1)\circ\Pi(\gamma_2).
\]
It is endpoint retaining if a homomorphism $p:\Sigma\to G$ satisfies $p(\Pi(T,\gamma))=\gamma(T)$. Surjectivity of $\Pi$ is not required. Write $\sigma=\Pi(T,\gamma)\in\Sigma$ for the synopsis element. The displayed square is the defining compatibility condition.
\end{definition}

\[
\begin{tikzcd}[column sep=large, row sep=large]
\mathcal{P}_e(G) \times \mathcal{P}_e(G)
\arrow[r, "\star"]
\arrow[d, "{\Pi \times \Pi}"']
&
\mathcal{P}_e(G)
\arrow[d, "{\Pi}"]
\\
{\Sigma \times \Sigma}
\arrow[r, "\circ"]
&
{\Sigma}
\end{tikzcd}
\]

 Neither the planar construction nor those below record the duration $T$. Duration can be retained by the Lie-group target $\Sigma\times(\mathbb R,+)$, using $(T,\gamma)\mapsto(\Pi(T,\gamma),T)$. The time coordinate adds under composition, and the image has nonnegative time coordinates. 

\section{Rotational Trajectories and their Concatenation} \label{pathsec}

Let $(G,\cdot)$ be a matrix Lie group (in the present context
$\SO3$) with operation $\cdot$ being matrix multiplication, and with identity element $e$ being an identity matrix of
appropriate dimension.
A \emph{based trajectory of duration $T\geq 0$} is a piecewise-$C^1$ map
\[
\gamma : [0,T] \to G
\]
satisfying
\[
\gamma(0) = e.
\]
The endpoint of the trajectory is $\gamma(T)\in G$.
Trajectories are allowed to have different durations.

\subsection{Composition}\label{sec:path-composition}

Let
\[
\gamma_1 : [0,T_1] \to G, \qquad
\gamma_2 : [0,T_2] \to G
\]
be based trajectories with $\gamma_1(0)=\gamma_2(0)=e$.
Define their \emph{raw concatenation} $\gamma_1 \star \gamma_2 : [0,T_1+T_2] \to G$ by
\[
(\gamma_1 \star \gamma_2)(t) =
\begin{cases}
\gamma_1(t), & 0 \le t \le T_1, \\[6pt]
\gamma_1(T_1)\cdot\gamma_2(t - T_1), & T_1 \le t \le T_1 + T_2 .
\end{cases}
\]

This definition ensures continuity at the splice point $t=T_1$, since
\[
\gamma_1(T_1) = \gamma_1(T_1) \cdot \gamma_2(0),
\]
and the concatenated trajectory remains based:
\[
(\gamma_1 \star \gamma_2)(0) = e.
\]
The endpoint satisfies
\[
(\gamma_1 \star \gamma_2)(T_1+T_2) = \gamma_1(T_1) \cdot \gamma_2(T_2),
\]
so concatenation corresponds to group multiplication of endpoints.

\subsection{Associativity}

Let
\[
\gamma_i : [0,T_i] \to G, \qquad i=1,2,3,
\]
be based trajectories with $\gamma_i(0)=e$.
Then both
\[
(\gamma_1 \star \gamma_2)\star \gamma_3
\quad \text{and} \quad
\gamma_1 \star (\gamma_2 \star \gamma_3)
\]
are defined on $[0,T_1+T_2+T_3]$.

A direct computation shows that both compositions are given by
\[
\gamma(t)=
\begin{cases}
\gamma_1(t), & 0 \le t \le T_1, \\[4pt]
\gamma_1(T_1) \cdot \gamma_2(t-T_1), & T_1 \le t \le T_1+T_2, \\[4pt]
\gamma_1(T_1) \cdot \gamma_2(T_2) \cdot \gamma_3(t-(T_1+T_2)),
& T_1+T_2 \le t \le T_1+T_2+T_3 .
\end{cases}
\]

Therefore, with this definition of $\star$,
\begin{equation}
((\gamma_1 \star \gamma_2)\star \gamma_3)(t)
=
(\gamma_1 \star (\gamma_2 \star \gamma_3))(t)
\label{asstraj}
\end{equation}
for all $t \in [0,T_1+T_2+T_3]$\,.

Hence, the set of duration-tagged based trajectories in $G$ with variable durations
forms a monoid under raw concatenation,
with duration addition and endpoint multiplication.

 The identity trajectory is the zero-duration path
$$\chg{\iota}(t) := e \,\,\,{\rm for}\,\,\, t \in \{0\}$$
 and it satisfies
$$ \gamma \star \chg{\iota} = \chg{\iota} \star \gamma = \gamma \,.$$

In the planar case we can keep track of the exact finite-dimensional synopsis resulting from the concatenation of rotational trajectories without needing
to keep the trajectories themselves. But in the spatial case, no single integer winding counter has the analogous role. However, a hierarchy of finite-dimensional descriptors of trajectories can be kept. Moreover, a full signature determines the reduced path only up to tree-like equivalence unless the trajectories satisfy further requirements and one keeps track of a strictly increasing time coordinate~\citep{HamblyLyons2010,LeeGhrist2020}. As such, a finite truncation of a signature on its own is not a complete description for arbitrary trajectories.

\section{Group Extensions Using Integral of Angular Velocity} \label{hybridsec}

In the planar case, total orientation could be obtained by simply integrating the angular rate, as stated in (\ref{intxeq}). In the 3D case this is not the case, but nevertheless information beyond a terminal orientation can be recorded by integrating angular velocity, which is a local property that can be computed directly from the trajectory in $\SO3$.

Let
\[
R(t)\in \SO3,
\qquad t\in[0,T],
\]
 be a based, piecewise-$C^1$ rotational trajectory with $R(0)=I$. Its angular velocities are piecewise continuous.  We consider both the body-fixed and
space-fixed angular velocities.

The body-fixed angular velocity is defined by
\[
\Omega_b(t)
=
R(t)^{-1}\dot R(t),
\]
while the space-fixed angular velocity is
\[
\Omega_s(t)
=
\dot R(t)R(t)^{-1}.
\]
Here $\omega_b=(R^{-1}\dot R)^\vee$, $\omega_s=(\dot R R^{-1})^\vee$, and $\Omega_\bullet=\omega_\bullet^{\wedge}$, so lowercase symbols are vectors and uppercase symbols are matrices.

In real-world applications $R(t)$ might not be known in advance and the
body-fixed angular velocity $\omega_b(t)$ is provided directly by gyroscopes, and $R(t)$ is then obtained by numerically solving the ordinary differential equation
\[
\dot{R}(t) = R(t) \,\Omega_b(t) \,.
\]
The space-fixed version of angular velocity is  obtained simply as
$$ \omega_s(t) = R(t) \, \omega_b(t) \,.$$

From these, it is simple to compute
\[
x_b(T)
:=
\int_0^T \omega_b(t)\,dt
\,\,\,\,{\rm and} \,\,\,\,
x_s(T)
:=
\int_0^T \omega_s(t)\,dt.
\]
These are both vectors in $\mathbb{R}^3$ and
$x_b^\wedge(T),x_s^\wedge(T)\in\mathfrak{so}(3)$.

First-order augmented orientation states can be defined as  $(R(T),x_b(T))$ or $(R(T),x_s(T))$ or $(R(T),x_b(T), x_s(T))$.
 The following subsection defines composition laws for two based path segments $\gamma_i:[0,T_i]\to\SO3$, with $R_i:=\gamma_i(T_i)$ and $x_{b,i}:=\int_0^{T_i}\omega_{b,i}(t)\,dt$; the subscript $i$ labels the segment, and the spatial quantities are defined analogously.

\subsection{Concatenation of Paths}

Consider two based rotational paths
\[
\gamma_1:[0,T_1]\rightarrow \SO3,
\qquad
\gamma_2:[0,T_2]\rightarrow \SO3,
\]
with endpoint rotations $R_1$ and $R_2$,
respectively. Their concatenation gives the end point
\[
R_{12}=R_1 R_2.
\]

For the body-fixed integral,
the concatenation law is additive:
\[
x_{b,12}
=
x_{b,1}+x_{b,2}.
\]
 The shift $\gamma_2(t-T_1)\mapsto R_1\gamma_2(t-T_1)$, with $R_1=\gamma_1(T_1)$, leaves $\Omega_b=R^{-1}\dot R$ and hence $\omega_b$ unchanged because the constant left factor cancels.

In contrast, for the space-fixed integral,
the second segment must be transported by the rotation generated by the
first segment.
The adjoint action on skew matrices is $\Ad_R(X)=RXR^{-1}$; under the hat identification, its vector form is $\Ad_Rx=Rx$.  Hence
\[
x_{s,12}
=
x_{s,1}
+
\operatorname{Ad}_{R_1} x_{s,2}.
\]
This is because the shift  $\gamma_2(t-T_1)\mapsto R_1\gamma_2(t-T_1)$ leads to the adjoint action
$R_1 \Omega_{s,2} R_1^T.$
Since for \(\SO3\) after converting to the spatial angular velocity vector
\[
\operatorname{Ad}_{R}x=Rx,
\]
this becomes
\[
x_{s,12}
=
x_{s,1}+R_1x_{s,2}.
\]

Therefore the combined first-order concatenation law is
\[
(R_1,x_{s,1},x_{b,1})
\circ
(R_2,x_{s,2},x_{b,2})
=
(
R_1R_2,\;
x_{s,1}+R_1x_{s,2},\;
x_{b,1}+x_{b,2}
).
\]

\begin{proposition}[First-order rotational synopsis]
The map $\Pi_1(\gamma)=(R(T),x_s(\gamma),x_b(\gamma))$ is an endpoint-retaining synopsis into $\Sigma_1=(\SO3\ltimes\mathbb R_s^3)\times\mathbb R_b^3$ with the displayed product.
\end{proposition}
\begin{proof}
On the second segment, $R_{12}=R_1R_2$, hence $\Omega_{b,12}=\Omega_{b,2}$ and $\Omega_{s,12}=R_1\Omega_{s,2}R_1^{-1}$. Splitting each integral at the splice gives the displayed product. Its identity is $(I,0,0)$ and its inverse is $(R^{-1},-R^{-1}x_s,-x_b)$. The projection $(R,x_s,x_b)\mapsto R$ recovers the endpoint, proving the commuting square for this construction.
\end{proof}

Equivalently, the corresponding group structure may be written as
\[
\left(
\SO3\ltimes\mathbb R^3_s
\right)
\times
\mathbb R^3_b.
\]

The semidirect product appears because the space-fixed integral transforms
under the adjoint action of \(\SO3\), while the body-fixed integral is
purely additive.

The endpoint orientation \(R(T)\) alone forgets the rotational history of the motion.  The augmented variables
\[
(R,x_b),
\qquad
(R,x_s),
\qquad
(R,x_s,x_b),
\]
retain coarse path information through integrated angular velocity.

These first-order variables may be viewed as continuous analogues of
lifted orientation variables or winding counters used in the planar theory of
total orientation.

\subsection{Matrix Representations of These Group Extensions}
In the following block matrices, $I=I_3$, $O=0_{3\times3}$, and $0_3$ is the zero column vector. The symbol $\oplus$ denotes block-diagonal concatenation. The group $\mathrm{SE}(3)$ consists of matrices $\left(\begin{smallmatrix}R&p\\0_3^T&1\end{smallmatrix}\right)$ with $R\in\SO3$ and $p\in\mathbb R^3$, under matrix multiplication.

The composition law for integrated body-fixed angular velocities can be reflected in matrix form as
\[
M_b(x_b)
=
\begin{pmatrix}
I & x_b\\
0_3^T & 1
\end{pmatrix}.
\]
Matrix multiplication yields
\[
M_b(x_{b,1})M_b(x_{b,2})
=
M_b(x_{b,1}+x_{b,2}).
\]
Including the joint information about rotation, we can write
$$ M_b(R, x_b) \,\doteq\, R \oplus M_b(x_b) =
\begin{pmatrix}
R & O & 0_3 \\
O & I & x_b\\
0_3^T & 0_3^T & 1
\end{pmatrix}.
$$
For the space-fixed case
\[
M_s(R,x_s)
=
\begin{pmatrix}
R & x_s\\
0_3^T & 1
\end{pmatrix}.
\]
Then
\[
M_s(R_1,x_{s,1})
M_s(R_2,x_{s,2})
=
M_s(
R_1R_2,
\;
x_{s,1}+R_1x_{s,2}
).
\]

 $M_s(R,x_s)$ is an $\mathrm{SE}(3)$ matrix. Set $p(t)=x_s(t)$ and define the auxiliary body translational velocity $v_b=R^T\dot p$. Then $\dot p=R\omega_b=\omega_s$, so the body twist, consisting of translational and angular velocities, is $(v_b,\omega_b)=(\omega_b,\omega_b)$. This is a mathematical embedding; $p$ does not represent a physical spacecraft translation. 

Suppose both integrated angular velocities are retained as $
(R,x_s,x_b)$.
A low-dimensional faithful representation is
$$  M_1(R,x_s,x_b) \doteq
\begin{pmatrix}
R & O & x_s \\
O & I & x_b\\
0_3^T &0_3^T & 1
\end{pmatrix}.
$$
The representation $M_1$ is $7\times7$.
 This representation satisfies
\[
M_1(R_1,x_{s,1},x_{b,1})
M_1(R_2,x_{s,2},x_{b,2})
=
M_1(
R_1R_2,\;
x_{s,1}+R_1x_{s,2},\;
x_{b,1}+x_{b,2}
)
\]
 and matches path concatenation.

The Magnus expansion extends these first-order descriptors by retaining higher homogeneous Lie terms~\citep{Magnus1954,Blanes2009}.

\section{The Magnus Expansion and Related Groups } \label{magnussec}

For a Lie-algebra-valued generator, the Magnus expansion expresses the logarithm of the time-ordered flow as homogeneous Lie polynomials~\citep{Magnus1954,Blanes2009}. The conventions and terms required for the subsequent constructions are given below.

 For $R(0)=I$, whenever the Magnus series converges, it expresses the solution of
\[
\dot R(t)=R(t) \Omega_b(t)
\,\,\,
{\rm and}
\,\,\,
\dot R(t)=\Omega_s(t)R(t)
\]
in the form
\[
R(T)=\exp(\Xi(T)),
\]
where $\Omega_b = \omega_b^{\wedge}$, $\Omega_s = \omega_s^{\wedge}$, and
\[
\Xi(T)
=
\Xi_1+\Xi_2+\Xi_3+\cdots
\]
is Lie-algebra valued.
For a general initial attitude, the body formula applies to $R(0)^TR(T)$ and the spatial formula to $R(T)R(0)^T$. Independently of convergence, each homogeneous Magnus coefficient is defined by its finite iterated-integral formula.

We use $0<t_1<\cdots<t_k<T$ throughout, so $t_1$ is the earliest time. In particular,
\[
\int_{0<t_1<t_2<T}F(t_1,t_2)\,dt_1dt_2
:=\int_0^T\int_0^{t_2}F(t_1,t_2)\,dt_1dt_2.
\]
The extension to higher nested integrals is analogous.
We write $[A,B]=AB-BA$ and $\operatorname{ad}_A(B)=[A,B]$. For the body equation $\dot R=R\Omega_b$, the first three terms are
\begin{align}
\Xi_{1,b}
&=\int_0^T\Omega_b(t)\,dt,\label{eq:body-magnus-first}\\
\Xi_{2,b}
&=\frac12\int_{0<t_1<t_2<T}
[\Omega_b(t_1),\Omega_b(t_2)]\,dt_1dt_2,\label{eq:body-magnus-second}\\
\Xi_{3,b}
&=\frac16\int_{0<t_1<t_2<t_3<T}\Big(
[\Omega_b(t_1),[\Omega_b(t_2),\Omega_b(t_3)]]\nonumber\\
&\hspace{42mm}+[\Omega_b(t_3),[\Omega_b(t_2),\Omega_b(t_1)]]
\Big)\,dt_1dt_2dt_3.\label{eq:body-magnus-third}
\end{align}
Writing the integration simplex in this order only relabels dummy variables. For example,
\[
-\frac12\int_{0<t_2<t_1<T}[\Omega(t_1),\Omega(t_2)]\,dt_2dt_1
=\frac12\int_{0<t_1<t_2<T}[\Omega(t_1),\Omega(t_2)]\,dt_1dt_2.
\]
The sign on the right follows from antisymmetry of the bracket. Reversing the inequalities alone, while keeping the integrand and prefactor fixed, would change the result. The third-order expression is unchanged under simultaneous exchange of $t_1$ and $t_3$, because its two summands interchange.

For the spatial equation $\dot R=\Omega_sR$, the first term has the same form, while the second term is $-\tfrac12\int_{0<t_1<t_2<T}[\Omega_s(t_1),\Omega_s(t_2)]\,dt_1dt_2$. This sign difference comes from the side on which the generator acts, not from the names of the integration variables. A general Magnus term is not a single right-nested commutator. A convention-safe recurrence is obtained by writing $\Xi=\sum_{n\geq1}\Xi_n$, setting $\dot\Xi_1=\Omega$, $\Xi_n(0)=0$, and, for $n\geq2$,
\begin{equation}
\dot\Xi_n=
\sum_{j=1}^{n-1}\frac{c_j}{j!}
\sum_{\substack{r_1+\cdots+r_j=n-1\\r_\ell\geq1}}
\operatorname{ad}_{\Xi_{r_1}}\cdots\operatorname{ad}_{\Xi_{r_j}}\Omega,
\end{equation}
 where $B_j$ are the Bernoulli numbers, defined by $z/(e^z-1)=\sum_{j\geq0}B_jz^j/j!$, with $B_0=1$ and $B_1=-\tfrac12$, $c_j=B_j$ for the spatial equation and $c_j=(-1)^jB_j$ for the body equation~\citep{Blanes2009}. Algebraic use of finite graded coordinates is independent of convergence of the evaluated infinite series; a standard sufficient condition is $\int_0^T\|\Omega(t)\|_2dt<\pi$~\citep{Blanes2009}.
Here $\|\Omega\|_2$ is the induced Euclidean operator norm; for $\Omega=\omega^\wedge\in\mathfrak{so}(3)$ it equals the Euclidean vector norm $\|\omega\|_2$.

For
\[
\mathfrak{so}(3)\cong\mathbb R^3,
\]
the Lie bracket becomes the cross product:

\[
[\Omega_1,\Omega_2]^{\vee} =\omega_1\times \omega_2.
\]

The Magnus expansion can be viewed as a reordering of terms in the definition of the state-transition matrix consistent with the Baker-Campbell-Hausdorff formula. Keeping only the first term yields exactly what was presented in the previous section. But keeping more terms builds a hierarchy of groups.

\subsection{Matrix Representation of the Body-Fixed Step-2 Magnus Group}

The Step-2 Magnus Groups are defined as follows.
As before, let
\[
x
=
\int_0^T \omega(t)\,dt
\,\,\,\,{\rm and}\,\,\,\,
b
=
\frac12
\int_{0<t_1<t_2<T}
\omega(t_1)\times\omega(t_2)
\,dt_1dt_2.
\]
With the chronological convention in \eqref{eq:body-magnus-second}, this $b$ equals $\Xi_{2,b}^{\vee}$ for the body equation.
The pair $(x,b)$
provides the step-2 Magnus coordinates.
The composition law for this is
\[
(x_1,b_1)\circ(x_2,b_2)
=
\left(
x_1+x_2,
\;
b_1+b_2+\frac12 x_1\times x_2
\right).
\]
 This is the free step-2 nilpotent Lie group on three generators (Section~\ref{sec:allorders}).

For the step-2 Magnus structure associated with
$\mathfrak{so}(3)$, the second-order bracket coordinate is naturally a vector
$b\in\mathbb R^3$,
or equivalently a skew-symmetric matrix
$b^\wedge\in\mathfrak{so}(3)$.

A dimensionally consistent realization is therefore obtained using block
sizes $1,\;3,\;3$ as
\[
\mathcal X(x,b)
=
\begin{pmatrix}
0 & 0 & 0\\
x & 0_{3\times3} & 0\\
b & \frac12{x}^{\wedge} & 0_{3\times3}
\end{pmatrix},
\]
where
\[
x,b\in\mathbb R^3.
\]
This is a \(7\times7\) nilpotent matrix.

Since
\[
\mathcal X^2=0,
\]
because its only potentially nonzero bottom-left block is $\tfrac12x^\wedge x=0$, the exponential truncates:
\[
\exp(\mathcal X)=I+\mathcal X.
\]

The corresponding group element is therefore
\[
\mathcal M(x,b)
=
I+\mathcal X(x,b)
=
\begin{pmatrix}
1 & 0 & 0\\
x & I_3 & 0\\
b & \frac12x^\wedge & I_3
\end{pmatrix}.
\]

Direct multiplication yields
\[
\mathcal M(x_1,b_1)\mathcal M(x_2,b_2)
=
\mathcal M
\left(
x_1+x_2,\;
b_1+b_2+\frac12 x_1\times x_2
\right).
\]

Thus matrix multiplication realizes the stated step-two law.

The second-order term may alternatively be represented either as a
vector $b\in\mathbb{R}^3$
or as a skew-symmetric matrix $
B\in\mathfrak{so}(3)$. The latter is more convenient for matrix representation, in which case
define $\mathcal M_B(x,B):=\mathcal M(x,B^\vee)$. Then
matrix multiplication yields

\[
\mathcal M_B(x_1,B_1)\mathcal M_B(x_2,B_2)
=
\mathcal M_B(
x_1+x_2,
\;
B_1+B_2+\frac12[x_1^\wedge,x_2^\wedge]
).
\]

For $\mathfrak{so}(3)$,
$[x_1^\wedge,x_2^\wedge]=(x_1\times x_2)^\wedge$.
The full matrix representation in the body-fixed case is then
$$ \mathcal M_B(R,x,B) = R \oplus \mathcal M_B(x,B)$$
with the simple direct-product law
\[
\mathcal M_B(R_1,x_1,B_1)\mathcal M_B(R_2,x_2,B_2)
=
\mathcal M_B(R_1 R_2,
x_1+x_2,
\;
B_1+B_2+\frac12[x_1^\wedge,x_2^\wedge]
).
\]
\subsection{Matrix Representation of the Space-Fixed Step-2 Transported Coordinate Group}

For the space-fixed formulation, it is convenient to distinguish between
$b\in\mathbb{R}^3$
and its associated skew-symmetric matrix
$B\in\mathfrak{so}(3)$,
where $B y=b\times y$.

For the space-fixed equation, let $b_s^{\mathrm{sig}}=-\Xi_{2,s}^{\vee}$ denote the chronological log-signature coordinate.
In the following product and matrices, $x=x_s$ and $b=b_s^{\mathrm{sig}}$. The space-fixed step-2  group law for this transported coordinate is

\[
(R_1,x_1,b_1)\circ_s(R_2,x_2,b_2)
=
\left(
R_1R_2,\;
x_1+R_1x_2,\;
b_1+R_1b_2+\frac12 x_1\times(R_1x_2)
\right).
\]

A convenient matrix representation with block sizes \(1,3,3\) is
\[
\mathcal M_s(R,x,b)
=
\begin{pmatrix}
1 & 0 & 0\\
x & R & 0\\
b & \frac12 {x}^{\wedge} R & R
\end{pmatrix}.
\]

Here $R\in\SO3$ and $x\in\mathbb R^3$.
Direct multiplication yields
\[
\mathcal M_s(R_1,x_1,b_1)
\mathcal M_s(R_2,x_2,b_2)
=
\mathcal M_s
\left(
R_1R_2,\;
x_1+R_1x_2,\;
b_1+R_1b_2+\frac12 x_1\times(R_1x_2)
\right).
\]

Thus the matrix representation exactly realizes the semidirect-product
structure of the space-fixed step-2 Magnus group.

The corresponding skew-matrix coordinate satisfies
$ B=b^{\wedge}$,
so that under transport by rotation,
\[
B
\mapsto
RBR^T.
\]

Hence the vector formulation transports by
\[
b\mapsto Rb,
\]
while the skew-matrix formulation transports by conjugation:
\[
B\mapsto RBR^T.
\]

If the actual spatial Magnus vector $\Xi_{2,s}^{\vee}$ is stored instead, its cross term is $-\tfrac12x_1\times(R_1x_2)$. This provides a hierarchy for constructing finite-dimensional Lie groups of increasing dimension, which capture higher levels of path information. A related framework is described in the next section.

\section{Iterated Integrals and Chen's Theory of Path Signatures} \label{chenrev}

Chen's signature of a vector-valued path is the collection of all iterated
integrals of that path~\citep{Chen1957,Chen1977}.  For an angular-velocity signal
$\omega:[0,T]\to\mathbb R^3$, define its integrated path $z(t)=\int_0^t\omega(s)ds$; the signature below is $S(z)$, with $dz=\omega(t)dt$, not an iterated integral against $d\omega$. Its signature is

\[
S(z)=(1,S_1(z),S_2(z),\ldots)\in\widehat T(V),
\qquad \widehat T(V):=\prod_{k=0}^{\infty}V^{\otimes k},\quad V=\mathbb R^3.
\]
Here $\widehat T(V)$ is the completed tensor algebra, with $V^{\otimes0}=\mathbb R$. Unlike an algebraic direct sum, it allows nonzero components at infinitely many degrees. Once a body or spatial angular-velocity convention is fixed, we write $S(\gamma):=S(z)$,
where $S_0(\gamma)=1$ and for $k>0$
$$ S_k(\gamma) = \int_{0<t_1<\cdots<t_k<T}
\omega(t_1)\otimes\cdots\otimes\omega(t_k)
\,dt_1\cdots dt_k. $$
Chen's identity~\citep{Chen1957,Chen1977} states that for ordinary concatenation of vector-valued
 paths and, for the based rotational concatenation used here, for body-fixed signatures,
\[
S(\gamma_1\star\gamma_2)=S(\gamma_1)S(\gamma_2),
\]
where multiplication is tensor multiplication.  The signature truncated at degree $N$ takes values in the finite-dimensional group of group-like elements of the truncated tensor algebra. The larger coordinate space is an ambient nilpotent group under truncated tensor multiplication, but not every arbitrary tuple of tensors is the signature of a path; genuine signatures satisfy the shuffle (group-like) relations.
These identities express a product of iterated integrals as the sum over interleavings that preserve the order within each factor. At truncation order $N$, only identities of total degree at most $N$ are imposed; \chg{Sections~\ref{sec:step2chen} and~\ref{sec:trunc-chen}} give examples and the precise group construction.

For rotational trajectories, one must first decide which angular-velocity
signal is being used.  With body-fixed angular velocity,
\[
S_b(\gamma)=S(z_b),\qquad z_b(t)=\int_0^t\omega_b(s)ds,
\]
and Chen's identity applies directly:
\[
S_b(\gamma_1\star\gamma_2)=S_b(\gamma_1)S_b(\gamma_2).
\]
With space-fixed angular velocity,
\[
S_s(\gamma)=S(z_s),\qquad z_s(t)=\int_0^t\omega_s(s)ds,
\]
but the second segment must be transported:
\[
S_s(\gamma_1\star\gamma_2)
=S_s(\gamma_1)\,\Ad_{R_1}S_s(\gamma_2),
\]
where $\Ad_{R_1}$ acts on each tensor factor.

{
The distinction between $x$ and $z$ is important. They are related as
$x = z(T)$, and $S_1 = x$, but the distinction can be seen starting with $S_2$:
$$ S_2(\gamma) = \int_{0}^{T} \int_{0}^{t_2} \omega(t_1) \,\omega^T(t_2)
\,dt_1 dt_2 =\int_{0}^{T} z(t_2)\,\omega^T(t_2) \,dt_2 \,.  $$
}

This distinction is essential.  The body-fixed signature is untwisted
under concatenation of based rotational paths, while the space-fixed
signature is twisted by the endpoint of the first segment.

A truncated body-fixed Chen synopsis of order $N$ is
\[
\mathcal S_b^{(N)}(\gamma)
=
\left(R(T),S_b^{(N)}(\gamma)\right),
\]
with multiplication
\[
(R_1,S_{b,1}^{(N)})(R_2,S_{b,2}^{(N)})
=
(R_1R_2,S_{b,1}^{(N)}S_{b,2}^{(N)}).
\]
Here the signature part multiplies directly because the body-fixed signal
concatenates directly.

A truncated space-fixed Chen synopsis of order $N$ is
\[
\mathcal S_s^{(N)}(\gamma)
=
\left(R(T),S_s^{(N)}(\gamma)\right),
\]
with multiplication
\[
(R_1,S_{s,1}^{(N)})(R_2,S_{s,2}^{(N)})
=
\left(R_1R_2,
S_{s,1}^{(N)}\,\Ad_{R_1}S_{s,2}^{(N)}
\right).
\]
The adjoint action is extended to tensors by
\[
\Ad_R(v_1\otimes\cdots\otimes v_k)
=(\Ad_Rv_1)\otimes\cdots\otimes(\Ad_Rv_k).
\]

The first-order laws are those of Section~\ref{hybridsec}; the second-order laws and their concrete matrix representations are developed below. Section~\ref{sec:allorders} proves the arbitrary-order statement, including closure, associativity, and inverses.

\subsection{Step-2 Chen Signatures}\chg{\label{sec:step2chen}}

Let \(\omega(t)\in\mathbb R^3\) denote either the body-fixed or
space-fixed angular velocity. Subscripts $b$ and $s$ on $x$ and $A$ mean that the following definitions use $\omega_b$ and $\omega_s$, respectively. The step-2 Chen signature is
\[
S^{(2)}(\omega)=(1,x,A),
\]
where
\[
x=\int_0^T \omega(t_1)\,dt_1\in\mathbb R^3,
\]
and
\[
A=
\int_{0<t_1<t_2<T}
\omega(t_1)\omega(t_2)^T\,dt_1dt_2
\in\mathbb R^{3\times3}.
\]

Equivalently,
\[
A=
\int_{0<t_1<t_2<T}
\omega(t_1)\otimes\omega(t_2)\,dt_1dt_2.
\]

Thus the Chen signature retains the second-order tensor coordinate
\(A\in\mathbb R^3\otimes\mathbb R^3\). For an actual path signature it obeys the shuffle constraint $A+A^T=xx^T$.

\subsection{Step-2 Chen Group Law}

For two step-2 signatures \((x_1,A_1)\) and \((x_2,A_2)\), concatenation
gives
\[
(x_1,A_1)\circ (x_2,A_2)
=
\left(
x_1+x_2,\;
A_1+A_2+x_1x_2^T
\right).
\]

The cross term \(x_1x_2^T\) comes from iterated integrals with one time in
the first segment and one time in the second segment, using the ordering
$0<t_1<t_2<T$.

Accordingly, the step-2 signature group is
\[
G_{\mathrm{sig}}^{(2)}(\mathbb R^3)=
\{(x,A):A+A^T=xx^T\},
\]
with the displayed product, identity $(0,0)$, and inverse $(-x,-A+xx^T)$. It is six dimensional. All pairs $(x,A)$ form a twelve-dimensional ambient nilpotent group under the same product, but only the constrained subgroup consists of group-like step-2 signatures.

\subsection{Relation to Step-2 Magnus Coordinates}

The step-2 Magnus coordinate keeps only the Lie part of the second-order
Chen tensor:
\[
B=-\frac12(A-A^T)\in\mathfrak{so}(3).
\]

Equivalently, if \(B=b^{\wedge}\), then
\[
b=-\frac12\,(A-A^T)^{\vee}\in\mathbb R^3.
\]

For genuine step-2 signatures, $A+A^T=xx^T$, so $A=\tfrac12xx^T-B$. Hence $(x,A)$ and the full step-2 log-signature coordinate $(x,B)$ carry exactly the same information; the larger twelve-dimensional space of arbitrary pairs $(x,A)$ is only an ambient group.

This equivalence is specific to the retained second-order information. Section~\ref{sec:third-distinction} shows that the full third-order Chen group has dimension $14$, whereas three evaluated Magnus vectors form a $9$-dimensional quotient group $M_3$. Both are groups, but evaluation discards five independent third-order free-Lie coordinates.

\subsection{Matrix Representation of the Step-2 Chen Group}

The step-2 Chen group law
\[
(x_1,A_1)\circ(x_2,A_2)
=
\left(
x_1+x_2,\;
A_1+A_2+x_1x_2^T
\right)
\]
can be represented by finite-dimensional matrices.

A dimensionally explicit scalar matrix representation uses
\[
\operatorname{vec}(A)\in\mathbb R^9.
\]
Explicitly, writing out $A$ as its columns, $A = [a_1,a_2,a_3]$, then
$$ \operatorname{vec}(A) =
\left(\begin{array}{c}
a_1 \\
a_2 \\
a_3 \end{array}\right) \,.
$$
With column-wise vectorization, define
the following matrix. Here $\otimes$ denotes the matrix Kronecker product, which represents the tensor product in the chosen bases; in particular, $(I_3\otimes x)y=\operatorname{vec}(xy^T)$:
\[
\mathcal C(x,A)
=
\begin{pmatrix}
1 & 0 & 0\\
x & I_3 & 0\\
\operatorname{vec}(A) & I_3\otimes x & I_9
\end{pmatrix},
\]
where the block sizes are \(1,3,9\).

Then, with a consistent vectorization convention,
\[
\mathcal C(x_1,A_1)\mathcal C(x_2,A_2)
=
\mathcal C
\left(
x_1+x_2,\;
A_1+A_2+x_1x_2^T
\right).
\]

\subsection{Matrix Representation of the Body-Fixed Full Group}

For the body-fixed convention, the orientation part and Chen part form
a direct product:
\[
\SO3\times G_{\mathrm{sig}}^{(2)}(\mathbb R^3).
\]

Thus one may write
\[
\mathcal M_b(R,x_b,A_b)
=
R\oplus \mathcal C(x_b,A_b),
\]
with multiplication
\[
\mathcal M_b(R_1,x_{b,1},A_{b,1})
\mathcal M_b(R_2,x_{b,2},A_{b,2})
=
\mathcal M_b
\left(
R_1R_2,\;
x_{b,1}+x_{b,2},\;
A_{b,1}+A_{b,2}+x_{b,1}x_{b,2}^T
\right).
\]

\subsection{Matrix Representation of the Space-Fixed Full Group}

For the space-fixed convention, the orientation acts on the Chen
coordinates by
\[
\alpha_R(x,A)=(Rx,RAR^T).
\]

Therefore the space-fixed group is the semidirect product
\[
\SO3\ltimes G_{\mathrm{sig}}^{(2)}(\mathbb R^3).
\]

The multiplication law is
\[
(R_1,x_1,A_1)(R_2,x_2,A_2)
=
\left(
R_1R_2,\;
x_1+R_1x_2,\;
A_1+R_1A_2R_1^T+x_1(R_1x_2)^T
\right).
\]

A matrix representation is obtained by allowing \(R\) to act by
conjugation on the Chen matrix representation:
\[
D(R)\mathcal C(x,A)D(R)^{-1}
=
\mathcal C(Rx,RAR^T).
\]

Using the \(\operatorname{vec}\)-form, this action matrix is
\[
D(R)
=
\begin{pmatrix}
1 & 0 & 0\\
0 & R & 0\\
0 & 0 & R\otimes R
\end{pmatrix}.
\]

With the convention
\[
\mathcal M_s(R,x,A)=\mathcal C(x,A)D(R),
\]
one obtains
\[
\mathcal M_s(R_1,x_1,A_1)
\mathcal M_s(R_2,x_2,A_2)
=
\mathcal M_s
\left(
R_1R_2,\;
x_1+R_1x_2,\;
A_1+R_1A_2R_1^T+x_1(R_1x_2)^T
\right).
\]

For direct use, the spatial representation just defined is
\[
\mathcal M_s(R,x,A)=
\begin{pmatrix}
1&0&0\\
x&R&0\\
\operatorname{vec}(A)&(I_3\otimes x)R&R\otimes R
\end{pmatrix}.
\]
It is a $13\times13$ representation of the ambient group and restricts to the genuine signature subgroup $A+A^T=xx^T$. Here column-wise vectorization gives $(I_3\otimes x)y=\operatorname{vec}(xy^T)$ and $\operatorname{vec}(RAR^T)=(R\otimes R)\operatorname{vec}(A)$. These identities establish the blocks without a second derivation.

The body and spatial constructions above retain the endpoint under direct and semidirect multiplication, respectively. Their extension and the limits of the Chen--Magnus equivalence are established next.

\section{Groups at Every Finite Truncation Order}\label{sec:allorders}

The second-order constructions extend to every integer $N\geq1$. The essential operation is to discard all terms of degree greater than $N$ in a graded algebra. This gives an exact finite-dimensional group, independently of whether an infinite Magnus series converges. We give the argument explicitly, including the endpoint rotation and the body/space convention. The tensor and free-Lie ingredients are classical~\citep{Chen1954,Chen1957,Reutenauer1993}; the purpose here is to establish the complete rotational synopsis construction and distinguish its two higher-order versions.

\subsection{Truncated Chen signatures}\chg{\label{sec:trunc-chen}}

Let $V=\mathbb R^3$ and define the truncated tensor algebra
\[
T^{\leq N}(V)=\bigoplus_{k=0}^N V^{\otimes k},\qquad
I_N=\bigoplus_{k=1}^N V^{\otimes k}.
\]

Here $V$ is the vector space of angular-velocity increments. The notation $V^{\otimes k}$ denotes the $k$-fold tensor product of $V$ with itself, with $V^{\otimes0}=\mathbb R$. The direct sum means that an element is a tuple
\[
 A=(A_0,A_1,\ldots,A_N),\qquad A_k\in V^{\otimes k}.
\]
Thus $A_0$ is a scalar, $A_1$ is a three-vector, $A_2$ has entries $(A_2)_{ij}$, and $A_k$ has $3^k$ entries $(A_k)_{i_1\cdots i_k}$. The subscript $k$ specifies tensor degree, not a vector component or a path-segment number. The symbol $\leq N$ means that degrees zero through $N$ are retained. The space $I_N$ consists of the tuples whose degree-zero component is zero. These spaces will serve, respectively, as the ambient algebra and its positive-degree ideal.

For two arbitrary tuples $A,B\in T^{\leq N}(V)$, define $A\odot B$ by tensor multiplication followed by deletion of degrees greater than $N$. This definition applies to all scalar components $A_0,B_0$; the restriction $A_0=B_0=1$ is imposed only when considering elements of the group $1+I_N$:
\par

\begin{equation}\label{eq:chen-all-product}
(A\odot B)_k=\sum_{i=0}^k A_i\otimes B_{k-i},\qquad 0\leq k\leq N.
\end{equation}
Here multiplication by the scalar degree-zero component has its usual meaning.
For example, at order two,
\[
 (1,a,P)\odot(1,b,Q)=(1,a+b,P+Q+a\otimes b).
\]
This is analogous to multiplying polynomials and retaining powers through $N$, except that tensor factors retain their order: generally $a\otimes b\ne b\otimes a$.
When $A$ is the signature of the integrated angular-velocity path $z(t)=\int_0^t\omega(s)\,ds$, its components are
\[
 A_0=1,\qquad A_k=\int_{0<t_1<\cdots<t_k<T}
 \omega(t_1)\otimes\cdots\otimes\omega(t_k)\,dt_1\cdots dt_k.
\]
The earliest time is on the left. The signal $\omega$ may be body-fixed or spatial, provided one convention is used consistently. General tuples in the tensor algebra need not be signatures.

The notation $1$ means $(1,0,\ldots,0)$, so $1+I_N$ consists of tuples $(1,u_1,\ldots,u_N)$. Powers below use the product just defined:
\[
 u^{\odot0}=1,\qquad u^{\odot j}=\underbrace{u\odot\cdots\odot u}_{j\text{ factors}}\quad(j\geq1).
\]
For $u\in I_N$, every factor has positive degree, so $u^{\odot(N+1)}=0$. More generally every product of $N+1$ elements of $I_N$ vanishes; this is the meaning of $I_N^{N+1}=0$.

The tensor exponential and logarithm are defined, for $L,u\in I_N$, by
\[
 \exp_{\odot}(L)=\sum_{j=0}^N\frac{L^{\odot j}}{j!},\qquad
 \log_{\odot}(1+u)=\sum_{j=1}^N\frac{(-1)^{j+1}}{j}u^{\odot j}.
\]
They are exact finite polynomials, not numerical approximations. For $v\in V$ placed in degree one, for example,
\[
 \exp_{\odot}(v)=(1,v,\tfrac12v\otimes v,\ldots,\tfrac1{N!}v^{\otimes N}).
\]

Embed $V$ in the degree-one component of $T^{\leq N}(V)$. The Lie algebra $\mathfrak f_N(V)$ generated by $V$ is the smallest linear subspace containing these vectors and closed under $[A,B]=A\odot B-B\odot A$. Concretely, it consists of linear combinations of vectors and nested brackets containing at most $N$ vectors. For example,
\[
 [v,w]=v\otimes w-w\otimes v,\qquad
 [v,[w,z]]=v\otimes w\otimes z-v\otimes z\otimes w
 -w\otimes z\otimes v+z\otimes w\otimes v,
\]
with expressions of degree above $N$ discarded. These are formal tensor brackets; no cross product has yet been applied. At order two, for a basis $e_1,e_2,e_3$ of $V$,
\[
 \mathfrak f_2(V)=\left\{v+\sum_{i<j}c_{ij}(e_i\otimes e_j-e_j\otimes e_i):
 v\in V,\ c_{ij}\in\mathbb R\right\}.
\]
In the present context, the word \emph{free} means that only the Lie identities (bilinearity, antisymmetry, and Jacobi) and the degree cutoff are imposed. The untruncated algebra is often written $\mathfrak f(V)$; the subscript $N$ specifies the truncation here. \emph{Step-$N$ nilpotent} means that brackets with more than $N$ entries vanish. The group $G_N(V)$ below is obtained by exponentiating this space of log-signatures, rather than all of $I_N$.
\par

\begin{theorem}[Chen groups at arbitrary order]\label{thm:chen-all}
The set $1+I_N$ is a finite-dimensional nilpotent Lie group under $\odot$, with identity $1=(1,0,\ldots,0)$ and inverse
\begin{equation}\label{eq:chen-all-inverse}
(1+u)^{-1}=\sum_{j=0}^N(-u)^{\odot j},\qquad u\in I_N.
\end{equation}
Let $\mathfrak f_N(V)$ be the Lie subalgebra of $T^{\leq N}(V)$ generated by $V$, with bracket $[A,B]=A\odot B-B\odot A$. Then
\[
G_N(V)=\exp_{\odot}(\mathfrak f_N(V))
\]
is the free step-$N$ nilpotent signature group. It contains every truncated Chen signature of the integrated angular-velocity path. Its product and inverse are the restrictions of \eqref{eq:chen-all-product} and \eqref{eq:chen-all-inverse}.
\end{theorem}

\begin{proof}
The tensors of degree greater than $N$ form a two-sided ideal in the full tensor algebra.
A two-sided ideal is a subspace stable under multiplication on either side by any algebra element. Here multiplication can only increase tensor degree. Passing to the quotient identifies all tensors of degree above $N$ with zero, which precisely implements truncation.
\par
 The quotient product is consequently associative, and $I_N^{N+1}=0$. Multiplying the finite geometric sum in \eqref{eq:chen-all-inverse} by $1+u$ on either side gives $1$. Multiplication and inversion are polynomial in the tensor entries, establishing the Lie-group assertion. A commutator of elements of positive degrees $i,j$ has degree at least $i+j$, which gives nilpotency of step at most $N$.

The finite polynomials
\[
\exp_{\odot}(L)=\sum_{j=0}^N\frac{L^{\odot j}}{j!},\qquad
\log_{\odot}(1+u)=\sum_{j=1}^N\frac{(-1)^{j+1}}{j}u^{\odot j}
\]
are mutually inverse. For $L,K\in\mathfrak f_N(V)$, the Baker--Campbell--Hausdorff (BCH) identity in this nilpotent algebra gives
the series, which expresses the logarithm of a product of exponentials,
\[
 \operatorname{BCH}(L,K)=L+K+\tfrac12[L,K]
 +\tfrac1{12}[L,[L,K]]+\tfrac1{12}[K,[K,L]]+\cdots.
\]
\par
\[
\exp_{\odot}(L)\odot\exp_{\odot}(K)
=\exp_{\odot}(\operatorname{BCH}_{\leq N}(L,K)).
\]
Here the subscript $\leq N$ retains total tensor degree at most $N$; since higher-degree terms vanish in the algebra, this truncation gives an exact group product.
Its right-hand logarithm is again in $\mathfrak f_N(V)$, and $\exp_{\odot}(L)^{-1}=\exp_{\odot}(-L)$. Thus $G_N(V)$ is a subgroup. It is closed because $\log_{\odot}$ is continuous and $\mathfrak f_N(V)$ is a finite-dimensional linear subspace.

For a constant angular-velocity segment with integrated increment $v$, direct integration gives $S^{(N)}=\exp_{\odot}(v)$. A piecewise-constant signal therefore has signature in $G_N(V)$ by Chen's identity. Approximating a piecewise-continuous angular velocity in $L^1$ by such signals makes every retained iterated integral converge; closedness then proves the assertion for the piecewise-$C^1$ rotational paths used here. Equivalently, these signatures satisfy the shuffle relations whose total degree is at most $N$. The free-Lie description of $\mathfrak f_N(V)$ is the standard one~\citep{Reutenauer1993}.
\end{proof}

In particular, the full third-order Chen coordinates $(1,x,A,C)$ obey the composition law~\eqref{eq:chen-third-product} below.
In this tuple $x\in V$, $A\in V^{\otimes2}$, and $C\in V^{\otimes3}$. The letter $A$ here denotes only the second-level tensor, unlike the generic full tuple $A$ used earlier. Subscripts $1,2$ in the following formulas label the two path segments; $12$ labels their concatenation.
\par

\begin{equation}\label{eq:chen-third-product}
\begin{aligned}
x_{12}&=x_1+x_2,\\
A_{12}&=A_1+A_2+x_1\otimes x_2,\\
C_{12}&=C_1+C_2+A_1\otimes x_2+x_1\otimes A_2.
\end{aligned}
\end{equation}
Thus retaining the third level does form a group; the genuine signature group is the subgroup satisfying the degree-two and degree-three shuffle constraints, not the whole space of arbitrary tensor entries.
Shuffle constraints are identities obtained by interleaving the time variables in products of iterated integrals while preserving their internal orders. For instance $A_{ij}+A_{ji}=x_i x_j$, or $A+A^T=xx^T$, at degree two. They explain why the tensor entries of a signature are not all independent.
\par

\subsection{Exact endpoints and spatial transport}
Here $\SO3=\{R\in\mathbb R^{3\times3}:R^TR=I,\det R=1\}$ is the rotation group. The map $\alpha_R$ below rotates every vector factor of a tensor and extends linearly to sums of tensors. A group automorphism is an invertible map preserving multiplication; thus $\alpha_R(S\odot T)=\alpha_R(S)\odot\alpha_R(T)$.

The symbols $b$, $s$, and $sb$ denote body, spatial, and combined (hybrid) coordinates. $S_b^{(N)}$ and $S_s^{(N)}$ are the signatures through degree $N$ of the corresponding integrated angular velocities. A direct product $\times$ multiplies its factors independently. A semidirect product $\ltimes_\alpha$ transports the second signature by the first rotation before multiplying:
\[
 (R_1,S_1)(R_2,S_2)=(R_1R_2,S_1\odot\alpha_{R_1}(S_2)).
\]
An endpoint-retaining synopsis stores $R(T)$ as well as the signature and converts path concatenation into this finite-dimensional multiplication.
\par

For $R\in\SO3$, define $\alpha_R$ on tensors by
\[
\alpha_R(v_1\otimes\cdots\otimes v_k)
=Rv_1\otimes\cdots\otimes Rv_k,\qquad \alpha_R(1)=1.
\]
This preserves tensor products, Lie brackets, and degree, hence restricts to a group automorphism of $G_N(V)$. Also $\alpha_{R_1R_2}=\alpha_{R_1}\alpha_{R_2}$.

\begin{theorem}[Endpoint-retaining Chen synopses]\label{thm:endpoint-all}
For every $N\geq1$, the body, spatial, and combined synopsis groups are
\[
\Sigma_{b,N}=\SO3\times G_N(V),\quad
\Sigma_{s,N}=\SO3\ltimes_{\alpha}G_N(V),\quad
\Sigma_{sb,N}=(\SO3\ltimes_{\alpha}G_N(V))\times G_N(V).
\]
In the combined coordinates their multiplication is
\begin{equation}\label{eq:hybrid-all-product}
\begin{aligned}
&(R_1,S_{s,1},S_{b,1})\circ(R_2,S_{s,2},S_{b,2})\\
&\quad=(R_1R_2,\ S_{s,1}\odot\alpha_{R_1}(S_{s,2}),\ S_{b,1}\odot S_{b,2}).
\end{aligned}
\end{equation}
The map $R(\cdot)\mapsto(R(T),S_s^{(N)},S_b^{(N)})$ is an endpoint-retaining synopsis. Dropping either signature coordinate gives the corresponding single-frame synopsis.
\end{theorem}

\begin{proof}
 Associativity follows from Theorem~\ref{thm:chen-all} and the automorphism property of $\alpha$. The combined identity is $(I,1,1)$ and
\[
(R,S_s,S_b)^{-1}=(R^{-1},\alpha_{R^{-1}}(S_s^{-1}),S_b^{-1}).
\]
These operations are smooth. On the second segment of based rotational concatenation, the body signal is unchanged and the spatial signal is multiplied by $R_1$. Chen's identity and the multilinearity of each iterated integral therefore give \eqref{eq:hybrid-all-product}. Projection onto $R$ returns the exact terminal rotation.  The theorem constructs the target groups; Proposition~\ref{prop:endpoint-adds} below shows that every element of each target is realized by a piecewise-$C^1$ based rotational path.
\end{proof}

\begin{proposition}[Endpoint independence at every finite order]\label{prop:endpoint-adds}
Fix $N\ge 1$. Let $\Pi_{b,N},\Pi_{s,N},\Pi_{sb,N}$ be the synopsis maps of
Theorem~\ref{thm:endpoint-all} on $\chg{\mathcal P}_e(\mathrm{SO}(3))$, with targets
$\Sigma_{b,N}=\mathrm{SO}(3)\times G_N(V)$,
$\Sigma_{s,N}=\mathrm{SO}(3)\ltimes_\alpha G_N(V)$ and
$\Sigma_{sb,N}=(\mathrm{SO}(3)\ltimes_\alpha G_N(V))\times G_N(V)$. Then:
\begin{enumerate}
\item[(i)] Each of $\Pi_{b,N},\Pi_{s,N},\Pi_{sb,N}$ is surjective. Piecewise-constant
body angular velocities of any prescribed total duration $T>0$ suffice. No bound on the angular-velocity magnitude is imposed.
\item[(ii)] Consequently, for every rotation $R\in\mathrm{SO}(3)$ and every value of
the signature coordinates (of $S_b^{(N)}$, of $S_s^{(N)}$, or of the pair
$(S_s^{(N)},S_b^{(N)})$), some path has endpoint $R$ and those signature values.
In particular, $R(T)$ is not determined by the body signature, the spatial
signature, or both together, truncated at any finite degree $N$.
\end{enumerate}
\end{proposition}

\begin{proof}
\emph{Step 0: Lie algebras and generators.}
Write $\mathfrak f_N=\bigoplus_{k=1}^N\mathfrak f_N^{(k)}$ (graded, $\mathfrak f_N^{(1)}=V$),
and identify $\mathfrak{so}(3)\cong V$ via $\wedge$. For $a\in V$ let $\rho(a)$ be
the derivation of $\mathfrak f_N$ with $\rho(a)v=a\times v$ on $V$; it is the
differential of $\alpha_{\exp(sa^\wedge)}$ and preserves degree. \chg{Set
$\mathfrak g_s=\mathfrak{so}(3)\ltimes_\rho\mathfrak f_N$}, with elements $(a,u)$ and
\[
[(a,u),(a',u')]=\bigl(a\times a',\ \rho(a)u'-\rho(a')u+[u,u']\bigr).
\]
The Lie algebras of $\Sigma_{b,N},\Sigma_{s,N},\Sigma_{sb,N}$ are
$\mathfrak g_b=\mathfrak{so}(3)\oplus\mathfrak f_N$, $\chg{\mathfrak g_s}$ and
$\mathfrak g_{sb}=\chg{\mathfrak g_s}\oplus\mathfrak f_N$. For $v\in V$ put
\[
Z_v=(v,v)\in\mathfrak g_b,\qquad Z_v=(v,v)\in\mathfrak g_s,\qquad Z_v=(v,v,v)\in\mathfrak g_{sb}.
\]
A segment with constant body rate $v$ and duration $\tau$ has $R(t)=\exp(tv^\wedge)$.
Since $\exp(tv^\wedge)v=v$, its spatial rate is also $v$. Hence its synopsis is
$\bigl(\exp(\tau v^\wedge),\exp_\odot(\tau v),\exp_\odot(\tau v)\bigr)$ (with the
coordinates appropriate to each group), which is $\exp(\tau Z_v)$.

\emph{Step 1: Reduction to Lie-algebra generation.}
Let $\mathfrak h$ be the Lie algebra generated by $\{Z_v:v\in V\}$. Because $\Pi$ is a
monoid homomorphism and $Z_{-v}=-Z_v$, the synopses of piecewise-constant
paths are exactly the finite products of one-parameter subgroups $\exp(sZ_v)$,
$s\in\mathbb R$.  \chg{By the orbit theorem~\citep{Sussmann1973}, applied to the symmetric family of left-invariant vector fields generated by the $Z_v$ (see also \citealp[Definition~2.3 and Lemma~2.5]{SussmannJurdjevic1972}), this set is the connected immersed Lie subgroup with Lie algebra $\mathfrak h$.} Each $\Sigma$ is connected, since
$\mathrm{SO}(3)$ is connected and $G_N\cong\mathfrak f_N$ as a manifold. \chg{So (i) holds as soon as $\mathfrak h$ is the full Lie algebra: the family is then bracket generating on the connected group $\Sigma$, and the Chow--Rashevskii theorem shows that every point of $\Sigma$ is reachable. This is the only direction used below.} The synopsis of a segment does
not depend on how increments are split between rate and duration, since
$\exp(\tau Z_v)=\exp(Z_{\tau v})$. Hence the total duration can be fixed to any $T>0$.

\emph{Step 2: Two algebraic facts.}
(F1) $\mathfrak{so}(3)$ and $\chg{\mathfrak g_s}$ are perfect ($[\mathfrak g,\mathfrak g]=\mathfrak g$),
and $\mathfrak f_N$ is nilpotent. For $\chg{\mathfrak g_s}$: the brackets $[(a,0),(a',0)]$
span $\mathfrak{so}(3)$; $[(a,0),(0,v)]=(0,a\times v)$ spans $V$; and
$[(0,u),(0,u')]$ spans $[\mathfrak f_N,\mathfrak f_N]=\bigoplus_{k\ge2}\mathfrak f_N^{(k)}$,
because $V$ generates $\mathfrak f_N$.

(F2) Let $\mathfrak a$ be perfect, $\mathfrak b$ nilpotent, and $\mathfrak h\subset\mathfrak a\oplus\mathfrak b$
a subalgebra with both projections onto. Then $\mathfrak h=\mathfrak a\oplus\mathfrak b$.
Indeed, $\mathfrak n=\{x:(x,0)\in\mathfrak h\}$ is an ideal of $\mathfrak a$, because
$[(x',y'),(x,0)]=([x',x],0)$ for any $(x',y')\in\mathfrak h$. The map
$\mathfrak h\to\mathfrak a/\mathfrak n$, $(x,y)\mapsto x+\mathfrak n$, is onto and kills
$\ker\mathrm{pr}_2=\mathfrak n\oplus0$, so it factors through
$\mathfrak h/\ker\mathrm{pr}_2\cong\mathfrak b$. Thus $\mathfrak a/\mathfrak n$ is perfect and
nilpotent, hence $0$. So $\mathfrak a\oplus0\subset\mathfrak h$, and since $\mathrm{pr}_2$
is onto, $\mathfrak h=\mathfrak a\oplus\mathfrak b$.

\emph{Step 3: Lemma A.} The subalgebra $\chg{\mathfrak h_s}\subset\chg{\mathfrak g_s}$ generated by
$\chg{Z_v}=(v,v)$ is all of $\chg{\mathfrak g_s}$.
Grade $\chg{\mathfrak g_s}$ by $\chg{\mathfrak g_s^{(0)}}=\mathfrak{so}(3)$ and $\chg{\mathfrak g_s^{(k)}}=\mathfrak f_N^{(k)}$,
$k\ge1$; brackets add degrees. Write $\chg{X_v}=(v,0)$ and $\chg{Y_v}=(0,v)$, so $\chg{Z_v=X_v+Y_v}$,
with $\chg{[X_a,X_b]=X_{a\times b}}$, $\chg{[X_a,Y_v]=Y_{a\times v}}$, and
$\chg{[Y_u,Y_v]}=(0,[u,v])$. Then
\[
\chg{[Z_u,Z_v]-Z_{u\times v}=Y_{u\times v}+[Y_u,Y_v]}\in\chg{\mathfrak h_s} .
\]
 For $0\ne x\in\chg{\mathfrak h_s}$, let $\operatorname{in}(x)$ denote its lowest-degree nonzero homogeneous component, and set
\[
\operatorname{gr}\chg{\mathfrak h_s}
=\operatorname{span}\{\operatorname{in}(x):0\ne x\in\chg{\mathfrak h_s}\}.
\]
The decreasing degree filtration induces
$F^k\chg{\mathfrak h_s}=\chg{\mathfrak h_s}\cap\bigoplus_{j=k}^N\chg{\mathfrak g_s^{(j)}}$
for $0\le k\le N+1$, with $F^{N+1}\chg{\mathfrak h_s}=0$.
Taking degree-$k$ components identifies $F^k\chg{\mathfrak h_s}/F^{k+1}\chg{\mathfrak h_s}$ with the degree-$k$ part of $\operatorname{gr}\chg{\mathfrak h_s}$; hence
$\dim\operatorname{gr}\chg{\mathfrak h_s}=\dim\chg{\mathfrak h_s}$.
It is a graded subalgebra because, whenever $[\operatorname{in}(x),\operatorname{in}(y)]\ne0$, this bracket equals $\operatorname{in}([x,y])$. Now
 $\operatorname{in}(\chg{Z_v})=\chg{X_v}$ for $v\ne0$ and,
for $u\times v\ne0$,
$\operatorname{in}\bigl(\chg{Y_{u\times v}+[Y_u,Y_v]}\bigr)=\chg{Y_{u\times v}}$.
Since cross products range over all of $V$, these initial components span all $\chg{X_a}$ and $\chg{Y_c}$.
The subalgebra generated by all $\chg{X_a,Y_c}$ is $\chg{\mathfrak g_s}$, so
$\dim\chg{\mathfrak h_s}=\dim\chg{\mathfrak g_s}$ and $\chg{\mathfrak h_s=\mathfrak g_s}$.

\emph{Step 4: Conclusion.}
Body: $\mathfrak h\subset\mathfrak{so}(3)\oplus\mathfrak f_N$ projects onto $\mathfrak{so}(3)$ (the $v^\wedge$ span it)
and onto $\mathfrak f_N$ ($V$ generates it); apply (F2) with $\mathfrak a=\mathfrak{so}(3)$.
Spatial: $\mathfrak h=\chg{\mathfrak h_s=\mathfrak g_s}$ by Lemma A.
Hybrid: $\mathfrak h\subset\chg{\mathfrak g_s}\oplus\mathfrak f_N$ projects onto $\chg{\mathfrak g_s}$ (Lemma A) and onto
$\mathfrak f_N$; apply (F2) with $\mathfrak a=\chg{\mathfrak g_s}$. This proves (i).
For (ii), fix a target $(R,\sigma)$; (i) supplies a path with that synopsis. Comparing
with the constant path, whose synopsis is the identity, shows that $R\ne I$ is compatible
with trivial signatures.
\end{proof}

\begin{remark}[Explicit witnesses; Magnus groups]
 For the body and spatial cases, explicit witnesses exist. Let $e_1,e_2,e_3$ be the standard coordinate vectors and use the group-commutator convention $ghg^{-1}h^{-1}$.
Let $g_i$ be the body synopsis of a constant-rate segment with increment $te_i$, and fix the nesting by
\[
\chg{W_1}=g_2,\qquad \chg{W_{k+1}}=g_1\chg{W_k}g_1^{-1}\chg{W_k}^{-1}\quad(1\le k\le N).
\]
The word $\chg{W_{N+1}}$ is realized by concatenating these segments; inverses are realized by reversing the segment order and negating the body rates.
It has \emph{exactly} trivial body signature in $G_N$, because all group commutators of length $N+1$ vanish in a step-$N$ nilpotent group. Its endpoint satisfies
 $\log R_t=t^{N+1}\operatorname{ad}^N_{e_1^\wedge}(e_2^\wedge)+O(t^{N+2})$.
Here $\operatorname{ad}^N$ denotes $N$ repeated brackets, and the remainder has norm bounded by a constant times $|t|^{N+2}$ near zero.
The leading coefficient is nonzero, since repeated cross products with $e_1$ alternate, up to sign, between $e_2$ and $e_3$. Thus $R_t\ne I$ for sufficiently small nonzero $t$.  The pointwise inverse path $s\mapsto R_t(s)^{-1}$, where $R_t(s)$ is the constructed path and $t$ is its increment parameter, has $\omega_s=-\omega_b$ of the original path, hence
an exactly trivial spatial signature and endpoint $R_t^{-1}\ne I$. The same commutator
does \emph{not} trivialize the spatial signature: for $N=1$ it gives
$x_s=2t^2e_3+O(t^3)$. This is why the hybrid case needs the argument above.
Since $M_N$ (and its spatial and hybrid versions) are quotients of the Chen groups
(Section~\ref{sec:third-distinction}), surjectivity descends, so the same conclusions hold
for the evaluated Magnus synopses.
\end{remark}

The retained rotation is therefore exact information in addition to the finitely truncated signature. Pairing it with the truncated signature is a direct/semidirect-product construction from standard homomorphisms; the preceding proposition identifies the information preserved by this choice of rotational descriptor.

Note that for piecewise-$C^1$ (bounded-variation) paths, $R(T)$ \emph{is} determined by the \emph{full} signature, since the signature fixes the path up to tree-like equivalence~\citep{HamblyLyons2010} and the development is invariant under it. Proposition~\ref{prop:endpoint-adds} therefore concerns finite truncations only, which do not
determine $R(T)$, and there is no need to do so since $R(T)$ is carried along with the finite truncation in this formulation.

\subsection{Evaluated Magnus coefficients as a graded group}

Evaluation replaces formal vector brackets by matrix commutators in $\mathfrak{so}(3)$, equivalently by cross products after applying the inverse hat map $\vee$. The formal symbol $\varepsilon$ below is a degree label, not a physical small parameter: $\varepsilon^{N+1}=0$ instructs us to discard degrees above $N$. Thus
\[
 L=\sum_{k=1}^N\varepsilon^k a_k^\wedge
\]
is a tuple of $N$ independent three-vectors $a_k$, with each degree kept separately. In the bracket formula below, $A,B$ denote skew-symmetric $3\times3$ matrices, not the tensor tuples used in \chg{Section~\ref{sec:trunc-chen}}. An algebra is \emph{graded} here because each component has a degree and brackets add degrees.

Retaining evaluated Magnus terms means retaining each homogeneous coefficient separately. Introduce a formal degree variable $\varepsilon$, with $\varepsilon^{N+1}=0$, and the Lie algebra
\begin{equation}\label{eq:graded-magnus-algebra}
\mathfrak m_N=\bigoplus_{k=1}^N\varepsilon^k\mathfrak{so}(3),\qquad
[\varepsilon^i A,\varepsilon^j B]=
\begin{cases}
\varepsilon^{i+j}[A,B],&i+j\leq N,\\
0,&i+j>N.
\end{cases}
\end{equation}
The degree tag is essential: the copies of $\mathfrak{so}(3)$ in different degrees are distinct vector spaces.

The notation $\mathfrak m_N$ refers to this vector space with its Lie bracket; $M_N$ refers to the same coordinate space equipped with the BCH group multiplication. The coefficient $\Xi_{k,b}(T)$ is the degree-$k$ body Magnus term, obtained by scaling the body angular-velocity matrix by a scalar and extracting the term homogeneous of degree $k$ in that scalar from the logarithm of the resulting development.

\begin{theorem}[Evaluated Magnus groups at arbitrary order]\label{thm:magnus-all}
The $3N$-dimensional space $\mathfrak m_N$ is a nilpotent Lie algebra. With multiplication
\[
L\circ_N K=\operatorname{BCH}_{\leq N}(L,K),
\]
where terms of total $\varepsilon$-degree greater than $N$ are discarded, it is a Lie group $M_N$ with identity $0$ and inverse $-L$. The body map
\[
R(\cdot)\longmapsto\left(R(T),\ \sum_{k=1}^N\varepsilon^k\Xi_{k,b}(T)\right)
\]
is an endpoint-retaining synopsis into $\SO3\times M_N$.
\end{theorem}

\begin{proof}
The bracket and Jacobi identity are inherited from the polynomial matrix Lie algebra after quotienting by terms of degree greater than $N$. Brackets increase degree, so every iterated bracket of length $N+1$ vanishes.

For an explicit faithful realization, let $J_N$ be the $(N+1)\times(N+1)$ matrix with ones on the first superdiagonal and zeros elsewhere.
Thus $J_N^{N+1}=0$. In the next display, $\otimes$ denotes the matrix Kronecker product: each entry of the first matrix multiplies a copy of the second matrix. The resulting $X_N(L)$ has size $3(N+1)\times3(N+1)$. A representation is \emph{faithful} if distinct group elements have distinct representing matrices; injectivity means the corresponding map loses no coordinates.
\par
 For $L=\sum_{k=1}^N\varepsilon^k a_k^\wedge$, set
\begin{equation}\label{eq:all-magnus-matrix}
X_N(L)=\sum_{k=1}^N J_N^k\otimes a_k^\wedge,\qquad
\mathcal E_N(L)=\exp X_N(L)=\sum_{j=0}^N\frac{X_N(L)^j}{j!}.
\end{equation}
The map $X_N$ is injective and preserves brackets. Its matrices are nilpotent, and the finite matrix logarithm is inverse to the displayed exponential. The BCH identity gives
\[
\mathcal E_N(L)\mathcal E_N(K)=\mathcal E_N(L\circ_N K).
\]
Associativity, the identity, and inverses now follow from matrix multiplication; the group operations are polynomial.

To identify the coefficients with the path descriptor, solve the formal equation $\dot U_\varepsilon=U_\varepsilon\varepsilon\Omega_b$, $U_\varepsilon(0)=I$, modulo $\varepsilon^{N+1}$. Its logarithm is exactly $\sum_{k=1}^N\varepsilon^k\Xi_{k,b}$. Concatenation gives $U_{\varepsilon,12}=U_{\varepsilon,1}U_{\varepsilon,2}$, so its logarithms compose by the displayed BCH product. The separate physical endpoint multiplies as $R_1R_2$. This formal argument involves finite polynomials and requires no Magnus convergence hypothesis.
\end{proof}

For example, in third-order body coordinates $L\leftrightarrow(x,b,c)$, the composition law is given by~\eqref{eq:magnus-third-product} below.
Specifically $L=\varepsilon x^\wedge+\varepsilon^2 b^\wedge+\varepsilon^3 c^\wedge$, with $x,b,c\in\mathbb R^3$. These are separate homogeneous log coefficients, not the raw second- and third-level signature tensors $A,C$. Subscripts $1,2,12$ again denote the first segment, second segment, and their concatenation.

\begin{equation}\label{eq:magnus-third-product}
\begin{aligned}
x_{12}&=x_1+x_2,\\
b_{12}&=b_1+b_2+\tfrac12x_1\times x_2,\\
c_{12}&=c_1+c_2+\tfrac12(x_1\times b_2+b_1\times x_2)\\
&\quad+\tfrac1{12}\bigl(x_1\times(x_1\times x_2)+x_2\times(x_2\times x_1)\bigr).
\end{aligned}
\end{equation}
This is an exactly associative group law on three retained vectors. Merely summing those vectors in $\mathfrak{so}(3)$ and applying a finite ordinary BCH approximation does not establish an associative product.

For spatial coordinates use the chronological convention
\begin{equation}\label{eq:spatial-magnus-all-sign}
a_{k,s}=(-1)^{k-1}\Xi_{k,s}^{\vee},\qquad
L_s=\sum_{k=1}^N\varepsilon^k a_{k,s}^{\wedge}.
\end{equation}
Indeed, the spatial differential equation reverses the chronological order of matrix words.
In particular, $a_{1,s}=\Xi_{1,s}^{\vee}$, $a_{2,s}=-\Xi_{2,s}^{\vee}$, and $a_{3,s}=\Xi_{3,s}^{\vee}$. This sign convention lets the stored spatial coefficients use the same chronological BCH product as the body coefficients. A matrix word is an ordered product of matrices; word reversal reverses that order.
\par
 Word reversal acts on a homogeneous Lie polynomial of degree $k$ by $(-1)^{k-1}$; this follows inductively because reversal changes a commutator into the negative commutator of the reversed words. Thus \eqref{eq:spatial-magnus-all-sign} extends the second-order sign convention already used above.

The action $\alpha_R L=\sum_k\varepsilon^k(Ra_k)^\wedge$ is an automorphism of $M_N$. Consequently, the spatial and combined evaluated Magnus groups are $\SO3\ltimes M_N$ and $(\SO3\ltimes M_N)\times M_N$, with the product in \eqref{eq:hybrid-all-product} interpreted using $\circ_N$ and $L_s,L_b$ in place of $\odot$ and $S_s,S_b$. Explicit matrices are
\[
\operatorname{diag}(R,\mathcal E_N(L_b)),\qquad
\mathcal E_N(L_s)\bigl(I_{N+1}\otimes R\bigr)
\]
for the body and spatial versions, respectively. Conjugation by $I_{N+1}\otimes R$ realizes $\alpha_R$; a block-diagonal combination with $\mathcal E_N(L_b)$ realizes the hybrid version. These uniform representations need not have minimal matrix size; the smaller step-two representations above remain useful.

\subsection{The third-order Chen--Magnus distinction}\label{sec:third-distinction}
A Lie-algebra homomorphism is a linear map preserving brackets. Write $E_i$ for the formal degree-one generator corresponding to $e_i$. Here $q_N$ replaces each formal generator by its physical skew matrix and keeps the degree tag; for example,
\[
 q_N([E_i,E_j])=\varepsilon^2[e_i^\wedge,e_j^\wedge]
 =\varepsilon^2(e_i\times e_j)^\wedge.
\]
In the present context, surjectivity means that every evaluated coefficient tuple has a preimage. The kernel consists of free-Lie expressions mapped to zero. A graded ideal is a degree-respecting subspace closed under taking brackets with any element of the algebra. Quotienting identifies elements differing by kernel information. The homogeneous dimension at degree $k$ counts independent log-signature coordinates involving exactly $k$ generators; it does not count all raw tensor entries.
\par

The full log signature is a free-Lie object before its brackets are evaluated in $\mathfrak{so}(3)$. Evaluation with degree retained defines the surjective Lie-algebra homomorphism
\[
q_N:\mathfrak f_N(\mathbb R^3)\longrightarrow\mathfrak m_N,
\qquad E_i\longmapsto\varepsilon e_i^\wedge.
\]
Surjectivity follows degree by degree from $[\mathfrak{so}(3),\mathfrak{so}(3)]=\mathfrak{so}(3)$. Its kernel is a graded ideal, so it induces the group quotient $G_N(\mathbb R^3)\to M_N$. Hence loss of information under evaluation does not prevent $M_N$ from being a group.

At order two the homogeneous free-Lie dimensions are $3,3$, and evaluation is an isomorphism, as reflected in $A=\tfrac12xx^T-b^\wedge$. At order three they are $3,3,8$, so the full third-order Chen group has dimension $14$, whereas $M_3$ has dimension $9$. The $39$ raw tensor entries $3+9+27$ should not be confused with the $14$ independent group coordinates. For an explicit element of the degree-three kernel, the nonzero free-Lie polynomial
\[
[E_1,[E_1,E_2]]-[E_3,[E_3,E_2]]
\]
evaluates to zero because $e_1\times(e_1\times e_2)=e_3\times(e_3\times e_2)=-e_2$. Thus the full third-order Chen group and the evaluated third-order Magnus group are both valid groups, but they are not equivalent descriptors. Including the exact endpoint gives target dimensions $17$ and $12$, respectively. This distinction qualifies the step-two equivalence and any higher-order reconstruction claim.

\section[Conditional Finite-Modal Reconstruction from Transported Moments]{ Conditional Finite-Modal Reconstruction from Transported Moments}\chg{\label{sec:modal}}

In general, the terms in the truncated Magnus and Chen expansions are insufficient to recover a trajectory in angular velocity, whether body-fixed or space-fixed.  This section establishes conditional finite-modal recovery from additional prescribed test-frame moments. It does not assert that the body or spatial truncations in the preceding sections alone determine those moments.

Suppose that the angular velocity is restricted to a ``modal'' form
\[
\omega_b(t)
=
\sum_{r=1}^{P}
c_r\phi_r(t),
\]
where
$
\phi_r:[0,T]\rightarrow\mathbb{R}^3
$
 are prescribed, linearly independent, piecewise-continuous basis functions on a known fixed interval $[0,T]$, and
$c_r\in\mathbb{R}$
are unknown coefficients. Such a restriction is reasonable when, for example, a motion is being designed or is the result of control actions.
Then the space of trajectories that results from integrating the kinematic equation to produce trajectories $R(t)$ for $t \in [0,T]$ is also $P$-dimensional.

The objective of this section is to construct a family of test trajectories that
 permit recovery of the coefficients \(c_r\) when the rank condition below holds, and hence exact reconstruction of \(\omega_b(t)\) and the rotation $R(t)$ obtained from $\dot R=R\omega_b^\wedge$, $R(0)=I$.

\subsection{Test Trajectories}

 Use one test frame per modal function, so the number of test frames is $m:=P$. For each $i=1,\ldots,m$, define a trajectory
$
g_i:[0,T]\rightarrow \SO3
$
as the solution of
\[
(g_i^{-1}\dot g_i)^\vee
=
\phi_i,
\qquad
g_i(0)=I.
\]

Associated with each trajectory is the adjoint action, represented in vector coordinates by
\[
Q_i(t)
=
\operatorname{Ad}_{g_i(t)}
\in \SO3.
\]
Equivalently, $g_i(t)v^\wedge g_i(t)^{-1}=(Q_i(t)v)^\wedge$ for every $v\in\mathbb R^3$.

We define the measurements
\[
a_i
=
\int_0^T
Q_i(t)
\omega_b(t)\,dt.
\]

These quantities may be regarded as generalized moments of the angular velocity trajectory, and under trajectory concatenation they behave as
\[
a_i(\gamma_1\star\gamma_2)
=a_{i,1}+Q_{i,1}(T_1)a_{i,2}.
\]
This identity assumes that the second test frame is rebased: $Q_{i,12}(T_1+s)=Q_{i,1}(T_1)Q_{i,2}(s)$. A reset copy of a general externally prescribed schedule does not satisfy this condition. For autonomous composition one must retain each terminal frame, with
\[
(Q_{i,1},a_{i,1})(Q_{i,2},a_{i,2})=(Q_{i,1}Q_{i,2},a_{i,1}+Q_{i,1}a_{i,2}).
\]
Consequently finite-dimensional synopsis laws can be constructed for jointly supplied physical and test-frame paths, with elements of the form $(R,(Q_i,a_i)_{i=1}^m)$. A synopsis of the physical path alone additionally requires a frame-assignment rule that itself respects concatenation.

\subsection{Linear Reconstruction}

Substituting
\[
\omega_b(t)
=
\sum_{r=1}^{P}
c_r\phi_r(t)
\]
into the measurement equation yields
\[
a_i
=
\sum_{r=1}^{P}
c_r
\int_0^T
Q_i(t)
\phi_r(t)
\,dt.
\]

Define
\[
M_{ir}
=
\int_0^T
Q_i(t)
\phi_r(t)
\,dt
\in\mathbb{R}^3.
\]

Then
\[
a_i
=
\sum_{r=1}^{P}
M_{ir}c_r.
\]

 Let $a=(a_1^T,\ldots,a_m^T)^T\in\mathbb R^{3m}$, $c=(c_1,\ldots,c_P)^T\in\mathbb R^P$, and let $M\in\mathbb R^{3m\times P}$ be the block matrix with vector blocks $M_{ir}$. Stacking the measurements gives
\[
a
=
Mc.
\]

If
\[
\operatorname{rank}(M)=P
\]
then the coefficients \(c_r\) are uniquely recoverable from the measurements $
a_i$.

Consequently, the trajectory $R(t)$ is uniquely determined by integrating its kinematic equation. This is a fixed-interval linear identifiability statement; it does not guarantee good conditioning or uniqueness for arbitrary basis and frame choices.  In particular, $c=M^\dagger a$, where $M^\dagger$ is the Moore--Penrose pseudoinverse. For fixed $M$ and perturbations only in the measurements, $\|\delta c\|_2\leq\|M^\dagger\|_2\|\delta a\|_2$.
One would attempt to design the set $\{\phi_i(t)\}$ such that
1) a rich set of trajectories in angular velocity results; and 2)
the rank condition is met. But if for some reason the rank is deficient,
the Magnus/Chen hierarchy affords the opportunity to go to higher terms.
For a simple full-rank example, take $T=1$, $P=m=2$, $\phi_1(t)=e_1$, and $\phi_2(t)=e_2$, where $e_1,e_2$ are standard coordinate vectors. The prescribed first test frame is $Q_1(t)=\exp(te_1^\wedge)$, and its two measurement columns are $M_{11}=e_1$ and $M_{12}=\sin(1)e_2+(1-\cos(1))e_3$. These columns are linearly independent, so the full stacked matrix has rank two and both modal coefficients are recoverable.

\subsection{Relation to Magnus and Chen Signatures}

The measurements \(a_i\) correspond to first-order information.

Indeed,
\[
a_i
=
\int_0^T
Q_i(t)
\omega_b\,dt
\]
 is the first Chen coefficient of the auxiliary integrated path $z_i(t)=\int_0^t Q_i(s)\omega_b(s)\,ds$: $a_i=S_1(z_i)$. It is also the first evaluated Magnus coefficient of that auxiliary signal. This interpretation does not identify $Q_i\omega_b$ with the relative angular velocity in a moving frame, which would in general include a frame-rate term.

These quantities contain only averaged information about the trajectory.

If the linear system is ill-conditioned or rank deficient, additional information
may be extracted from higher signature or log-signature terms, although these do not guarantee injectivity without a separate argument.

\subsection{Second-Order Log-Signature Measurements}

Define
\[
b_{i,2}^{\mathrm{sig}}
=
\frac12
\int_{0<t_1<t_2<T}
\big(Q_i(t_1)\omega_b(t_1)\big)\times
\big(Q_i(t_2)\omega_b(t_2)\big)
dt_1dt_2.
\]

Substituting the expansion
\[
\omega_b
=
\sum_{r=1}^{P}
c_r\phi_r
\]
yields
\[
b_{i,2}^{\mathrm{sig}}
=
\frac12
\sum_{r,s}
c_rc_s
K^{(2)}_{irs},
\]
where
\[
K^{(2)}_{irs}
=
\int_{0<t_1<t_2<T}
\big(Q_i(t_1)\phi_r(t_1)\big)\times
\big(Q_i(t_2)\phi_s(t_2)\big)
dt_1dt_2.
\]

Thus the second log-signature coefficient supplies quadratic constraints on the
unknown coefficients. Higher orders may add constraints, but no general recovery claim follows when the rank condition fails; for example, collinear signals have vanishing commutators at every order.

\section{Geometric Phase: Underactuated Satellite}\label{sec:geophase}

{
Geometric phase is the residual motion accumulated when a system traverses a closed path in a reduced or shape space, after the dynamic contribution has been separated out~\citep{Montgomery1991}.
In systems with symmetry, returning the reduced state to its initial value need not return the full configuration to its initial value.
The resulting displacement depends on the geometry of the reduced path, providing a mechanism for generating motion through cyclic changes of the reduced variables~\citep{KellyMurray1995}.

For an axisymmetric satellite with two transverse control torques and zero initial axial spin, a closed loop of the symmetry-axis direction can generate a net rotation about that axis, even though the instantaneous axial angular velocity remains zero.
This geometric phase can therefore be used to generate the axial rotation needed to complete an attitude maneuver after pointing alignment.
Related spacecraft steering methods have also been developed without explicitly invoking geometric phase.
Using an alternative formulation of the attitude kinematics, \citet{Tsiotras1995} derive a stabilizing feedback law for arbitrary reorientation with two transverse control torques and zero initial axial spin.
\citet{TsiotrasLuo2000} develop stabilization and tracking laws with bounded kinematic inputs, together with feasible trajectory generation based on differential flatness.

Satellite steering by geometric phase provides a natural setting for interpreting the path information retained by the proposed synopsis groups.
We use this example to connect the algebraic description of rotational history with the physical effect of cyclic pointing motion.
\par}

\subsection{Reduced Dynamics of the Axisymmetric Satellite}

\paragraph{System Dynamics and Momentum Conservation}

Consider the attitude dynamics of an underactuated satellite that is inertially symmetric about the third body-fixed axis. Its dynamics are governed by
\begin{gather}
	J\dot\omega_b+\omega_b\times J\omega_b = \tau,\label{eqn:dot_omega_b}\\
    \dot R = R\omega_b^\wedge,
\end{gather}
where \(J=\mathrm{diag}[J_1,J_1,J_3]\in\mathbb R^{3\times 3}\)  is the inertia tensor, with $J_1,J_3>0$, and $\tau=[\tau_1,\tau_2,0]^T\in\mathbb R^3$ is the underactuated control moment.

{
Let \(e_3=[0,0,1]^T\in\mathbb R^3\).
Applying the axial symmetry condition \(J_1=J_2\) and the underactuation constraint \(e_3^T\tau=0\) yields conservation of the angular momentum about the symmetry axis.
Specifically, we obtain $e_3^T J\dot\omega_b = 0$, which implies that the third component of the angular velocity is conserved.
\begin{equation}\label{eqn:axial_momentum}
    \mu := J_3e_3^T\omega_b(t)=J_3e_3^T\omega_b(0),
\end{equation}
for all \(t\geq 0\).
 Thus $e_3^T\omega_b=\mu/J_3$ is constant. In the zero-axial-momentum case $\mu=0$, the angular velocity lies in the plane orthogonal to $e_3$; for nonzero $\mu$, it lies in an affine plane offset from the origin.
The system is underactuated because the rotation about the axis of symmetry cannot be controlled directly by the applied control moment.
\par}

\paragraph{ Reduction and Reconstruction}
Write $S^2=\{v\in\mathbb R^3:\|v\|=1\}$ and identify $S^1$ with angles modulo $2\pi$. The right action $R\mapsto R\exp(\theta e_3^\wedge)$ leaves $Re_3$ fixed; all attitudes with a given $Re_3=\Gamma$ form the fiber over $\Gamma$.
This is standard Lagrange--Routh reduction and reconstruction~\citep{marsden2000reduction}: the $S^1$-symmetry about the axisymmetric spacecraft's body axis gives a principal bundle $\SO3\to\Sph^2$, $\pi(R)=Re_3$, with base variable $\Gamma=Re_3$ and the above momentum map $\mu=J_3e_3^T\omega_b$. The mechanical connection is
\begin{equation}
\mathcal A_R(\dot R)=e_3^T(R^T\dot R)^\vee=e_3^T\omega_b.
\end{equation}
Its horizontal motions have zero instantaneous spin about the symmetry axis.

Given the value of $\mu$, the reduced dynamics on $\Sph^2$ is governed by
\begin{equation}
	J_1(\ddot\Gamma+\|\dot\Gamma\|^2\Gamma) = \mu\Gamma\times\dot\Gamma - \Gamma\times\tau_s,\label{eqn:Gamma_dot}
\end{equation}
where $\tau_s = R \tau\in\mathbb R^3$ corresponds to the control moment resolved in the inertial frame, and it is orthogonal to $\Gamma$, i.e., $\Gamma^T\tau_s=0$.

\paragraph{Reconstruction}
For a prescribed reduced curve $\Gamma(t)$, its horizontal lift is defined by
\begin{equation}
\dot R_{\mathrm{hor}}=(\Gamma\times\dot\Gamma)^\wedge R_{\mathrm{hor}},
\qquad R_{\mathrm{hor}}(0)=R(0).
\end{equation}
It satisfies $R_{\mathrm{hor}}e_3=\Gamma$ and $\mathcal A_{R_{\mathrm{hor}}}(\dot R_{\mathrm{hor}})=0$. The full trajectory is reconstructed as
\begin{equation}
	R(t) = R_{\mathrm{hor}}(t)\exp(\theta(t)e_3^\wedge),\label{eqn:R}
\end{equation}
for $\theta(t)\in\Sph^1$ in the fiber phase { corresponding to the angular displacement relative to the horizontal lift.}
The conservation of the momentum yields the evolution of $\theta(t)$ as
\begin{equation}
	\theta(t) = \frac{\mu}{J_3}t,
\end{equation}
which can be substituted into \eqref{eqn:R} to obtain the complete trajectory in $\SO3$.

\subsection{Geometric Phase and Mechanical Connection}

Suppose that the reduced trajectory \(\Gamma(t)\) on \(t\in[0,T]\) is closed, i.e., $\Gamma(0)=\Gamma(T)$.
Since $R_{\mathrm{hor}}(t)e_3=\Gamma(t)$, the initial and terminal points of the horizontal lift belong to the same fiber.
Therefore, there exists an angle \(\theta_{\mathrm{geo}}\in\mathbb{R}\) such that
\begin{equation}
    R_{\mathrm{hor}}(T) = R(0)\exp(\theta_{\mathrm{geo}}e_3^\wedge).
\end{equation}
Substituting this into the reconstruction equation \eqref{eqn:R} at $t=T$, we obtain
\begin{equation}
	R(T) = R(0) \exp(\theta_{\mathrm{geo}}e_3^\wedge) \exp(\theta_{\mathrm{dyn}}e_3^\wedge), \label{eqn:RT}
\end{equation}
where $ \theta_{\mathrm{dyn}} = \theta(T) = \frac{\mu}{J_3}T$.
The quantities \(\theta_{\mathrm{dyn}}\) and \(\theta_{\mathrm{geo}}\) are referred to as the dynamic phase and geometric phase, respectively.
When \(\mu=0\), the dynamic phase vanishes and the entire rotation about the symmetry axis is generated by the geometric phase.

The geometric phase is equal to the holonomy of the mechanical connection around the closed reduced trajectory:
\begin{equation}
    \theta_{\mathrm{geo}} \equiv -\int_{\mathcal{F}} \mathcal K\pmod{2\pi},
\end{equation}
 where $\mathcal F$ is an oriented spanning surface, or more generally an oriented 2-chain on $S^2$, whose boundary is the parametrized closed curve $\Gamma(t)$ with its orientation and multiplicity, and $\mathcal K$ is the curvature two-form on $S^2$, characterized by $\pi^*\mathcal K=d\mathcal A$.
Here $d$ denotes exterior differentiation and $\pi^*$ denotes pullback. We use the outward orientation of $S^2$ and the induced boundary orientation. A real-valued phase requires a choice of representative; the holonomy itself fixes it modulo $2\pi$.

For the axisymmetric spacecraft considered here, the curvature two-form is equal to the negative of the standard area form on \(S^2\).
 Consequently, with this orientation convention, the geometric phase equals the signed spherical area enclosed by the reduced trajectory modulo $2\pi$.

\subsection{Magnus Synopsis Variables and Geometric Phase}\chg{\label{sec:magnus-geophase}}

\paragraph{Local Relation Between Geometric Phase and Synopsis Variables}

Choose $R(0)=I$, so $\Gamma(0)=e_3$, and let the reduced trajectory remain in a sufficiently small neighborhood of $e_3$. Consider a controlled family of closed loops on a fixed time interval whose coordinates and first derivatives are uniformly $O(\chg{\delta})$. Introduce local coordinates
\begin{equation*}
    \Gamma(t)=\begin{bmatrix}x(t)&y(t)&\sqrt{1-x(t)^2-y(t)^2}\end{bmatrix}^T,
\end{equation*}
where \(x,y=O(\chg{\delta})\).
The area form on \(S^2\) is
\begin{equation}
\omega_{S^2}=\frac{1}{\sqrt{1-x^2-y^2}}\,dx\wedge dy = (1+\mathcal{O}(\chg{\delta}^2))dx\wedge dy.
\end{equation}
Since $d(x\,dy-y\,dx)=2\,dx\wedge dy$, applying Stokes' theorem yields
\begin{equation}
	\theta_{\mathrm{geo}}=\frac12\oint_{\Gamma}(x\,dy-y\,dx)+\mathcal{O}(\chg{\delta}^4).
\end{equation}

On the other hand, the second-order body-fixed Magnus synopsis variable is
\begin{equation}
    b=\frac12\int_{0<t_1<t_2<T}\omega_b(t_1)\times\omega_b(t_2)\,dt_1dt_2.
\end{equation}
For the horizontal lift with this initial attitude,
$\omega_b=R_{\mathrm{hor}}^T(\Gamma\times\dot\Gamma)
=[-\dot y,\dot x,0]^T+O(\chg{\delta}^3)$.
Indeed, the horizontal lift differs from the shortest rotation carrying $e_3$ to $\Gamma$ by an axial angle of order $\chg{\delta}^2$; its effect on the transverse velocity is therefore of order $\chg{\delta}^3$. This estimate justifies the order-$\chg{\delta}^4$ remainder in the quadratic integral below.
Substituting \(\omega_{b_1}=-\dot y\) and \(\omega_{b_2}=\dot x\), the third component of $b$ is
\begin{equation*}
    b\cdot e_3=\frac12\int_0^T\int_0^{t_2}\left(\dot x(t_1)\dot y(t_2)-\dot y(t_1)\dot x(t_2)\right)\,dt_1dt_2+\mathcal O(\chg{\delta}^4).
\end{equation*}
Evaluating the inner integrals yields
\begin{equation*}
    b\cdot e_3=\frac12\int_0^T\left((x(t)-x(0))\dot y(t)-(y(t)-y(0))\dot x(t)\right)\,dt+\mathcal O(\chg{\delta}^4).
\end{equation*}
Since the curve is closed, \(\int_0^T\dot x(t)\,dt=0\) and \(\int_0^T\dot y(t)\,dt=0\), so the terms involving \(x(0)\) and \(y(0)\) vanish. Therefore
\begin{equation*}
    b\cdot e_3=\frac12\int_0^T\left(x(t)\dot y(t)-y(t)\dot x(t)\right)\,dt+\mathcal O(\chg{\delta}^4).
\end{equation*}
Finally, using \(dx=\dot x\,dt\) and \(dy=\dot y\,dt\), this becomes
\begin{equation}
    b\cdot e_3=\frac12\oint_{\Gamma}(x\,dy-y\,dx)+\mathcal O(\chg{\delta}^4).
\end{equation}

Consequently, the second-order synopsis variable captures the leading-order (flat-space) approximation of the geometric phase.
Thus, under the above orientation and Magnus-sign conventions,
\begin{equation}
    \theta_{\mathrm{geo}} = b\cdot e_3+\mathcal O(\chg{\delta}^4).
\end{equation}
Therefore, the second-order Magnus synopsis variable captures the leading-order, local approximation of the geometric phase associated with the mechanical connection on \(\SO3\to \Sph^2\).

\paragraph{Geometric Phase and Higher-Order Magnus Synopsis Variables}

When the Magnus series converges, the higher-order synopsis variables capture the geometric phase in the limit.
The Magnus expansion provides
\[
    R(0)^T R(T) = \exp\!\left(\Xi_b(T)\right) = \exp\!\left((x_b+b_b+c_b+\cdots)^\wedge\right),
\]
where $x_b=\Xi_{1,b}^\vee$, $b_b=\Xi_{2,b}^\vee$, $c_b=\Xi_{3,b}^\vee$, $d_b=\Xi_{4,b}^\vee$, etc.

If the reduced trajectory is closed and \(\mu=0\), then the dynamic phase vanishes and
\[
    R(0)^T R(T)=\exp(\theta_{\mathrm{geo}}e_3^\wedge).
\]
Hence
\[
    \exp\!\left((x_b+b_b+c_b+\cdots)^\wedge\right) = \exp(\theta_{\mathrm{geo}}e_3^\wedge).
\]

For this identification assume that the convergent Magnus logarithm is in the principal rotation-vector domain $\|\Xi_b^\vee(T)\|<\pi$; for example, the sufficient condition $\int_0^T\|\omega_b(t)\|dt<\pi$ ensures this local regime. Choosing the corresponding representative of the phase then gives
\[
    \theta_{\mathrm{geo}} = e_3^T(x_b+b_b+c_b+\cdots).
\]
The branch restriction matters: equality of two matrix exponentials alone does not justify equality of the axial projections of arbitrary logarithms, even modulo $2\pi$.

However, the odd-order Magnus terms have zero component along $e_3$. Indeed, a homogeneous order-$n$ nested bracket contains $n-1$ cross products, and brackets alternate between the $e_1$--$e_2$ plane and the $e_3$ axis.
Consequently, the geometric phase is determined only by the \(e_3\)-components of the even-order Magnus terms; equivalently,
\begin{equation}
    \theta_{\mathrm{geo}} \equiv e_3^T(b_b+d_b+\cdots) \pmod{2\pi}.
\end{equation}

{
For a closed reduced loop with zero axial momentum, the retained endpoint attitude determines the geometric phase modulo $2\pi$.
The purpose of the Magnus coordinates here is to describe how contributions at different orders combine to produce this phase.
This interpretation connects the synopsis coordinates to a physical consequence of rotational history, while the endpoint provides an exact reference for assessing finite-order phase approximations.
\par}

\subsection{Satellite Steering by Geometric Phase}

\begin{figure}
	\begin{center}
		\footnotesize
		\tdplotsetmaincoords{65}{115}

		\begin{tikzpicture}[
			tdplot_main_coords,
			scale=4,
			line cap=round,
			line join=round,
			>=Stealth
			]
			\def\R{1}
			\def\psi{55}   %
			\def\rho{35}   %
			\def\sig{1}    %

			\coordinate (O)  at (0,0,0);
			\coordinate (Gd) at (0,0,1);
			\coordinate (G0) at (0,{sin(\psi)},{cos(\psi)});
			\coordinate (uone) at (1,0,0);
			\coordinate (utwo) at ({\sig*sin(\rho)},0,{cos(\rho)});

			\tikzset{
				spheregrid/.style={gray!35, thin},
				phaseone/.style={
					very thick,
					blue!70!black,
					postaction={
						decorate,
						decoration={
							markings,
							mark=at position 0.55 with {\arrow{Stealth}}
						}
					}
				},
				phasetwo/.style={
					very thick,
					blue!70!black,
					postaction={
						decorate,
						decoration={
							markings,
							mark=at position 0.25 with {\arrow{Stealth}},
							mark=at position 0.70 with {\arrow{Stealth}}
						}
					}
				},
				vec/.style={-{Stealth[length=2mm]}, thick},
				point/.style={circle, fill=black, inner sep=0.8pt}
			}

		\begin{scope}[tdplot_screen_coords]
				\shade[ball color=blue!8, opacity=0.20] (0,0) circle (\R);
			\draw[gray!60, thin] (0,0) circle (\R);
		\end{scope}

		\foreach \z in {-0.6,-0.3,0,0.3,0.6}{
			\pgfmathsetmacro{\rr}{sqrt(1-\z*\z)}
			\draw[spheregrid]
				plot[domain=0:360, samples=90, variable=\t]
				({\rr*cos(\t)},{\rr*sin(\t)},{\z});
		}

		\foreach \ph in {0,30,60,90,120,150}{
			\draw[spheregrid]
				plot[domain=-90:90, samples=80, variable=\t]
				({cos(\t)*cos(\ph)},{cos(\t)*sin(\ph)},{sin(\t)});
		}

		\draw[phaseone]
			plot[domain=0:\psi, samples=80, variable=\t]
			({0},{sin(\psi-\t)},{cos(\psi-\t)});

		\node[blue!70!black, anchor=south west] at (0,{sin(0.55*\psi)},{cos(0.55*\psi)})
			{\small Phase 1};

		\draw[phasetwo]
			plot[domain=0:360, samples=160, variable=\a]
			(
			{\sig*sin(\rho)*cos(\rho)*(1-cos(\a))},
			{-sin(\rho)*sin(\a)},
			{cos(\rho)*cos(\rho)+sin(\rho)*sin(\rho)*cos(\a)}
			);

		\node[anchor=north east] at
			({\sig*sin(\rho)*cos(\rho)*(1-cos(155))},
			{-sin(\rho)*sin(155)},
			{cos(\rho)*cos(\rho)+sin(\rho)*sin(\rho)*cos(155)})
			{\small Phase 2};

		\draw[vec] (O) -- (uone)
			node[anchor=north east] {$u_1$};

		\draw[vec] (O) -- (utwo)
		node[anchor=south west] {$u_2$};

		\draw[vec] (O) -- (G0)
		node[anchor=south west] {$\Gamma_0$};

		\draw[vec] (O) -- (Gd)
		node[anchor=south] {$\Gamma_d$};

		\node[point] at (G0) {};
		\node[point] at (Gd) {};
		\node[point] at (utwo) {};

		\node[red!75!black, anchor=west] at
		({0.55*\sig*sin(\rho)},0,{0.55*cos(\rho)})
		{};

	\end{tikzpicture}

\end{center}
\caption{Satellite steering using geometric phase: visualization of the trajectory of the axis of symmetry $\Gamma=Re_3$ on the unit-sphere}
\label{fig:sphere_trajectory}
\end{figure}
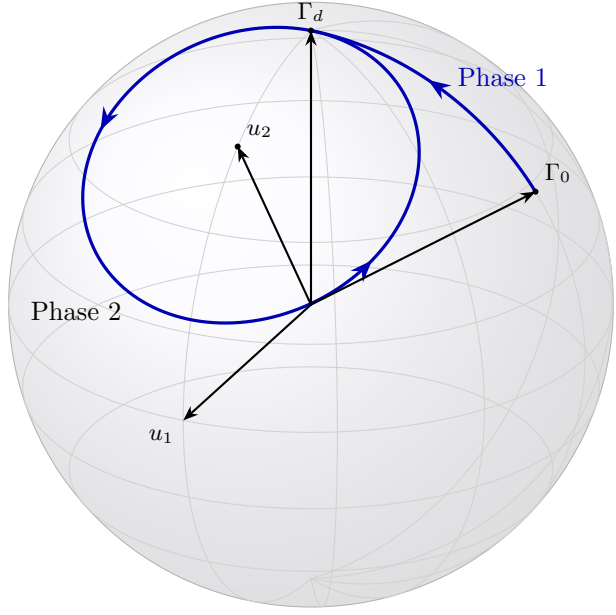

The geometric phase presented above can be used for motion planning of an underactuated satellite.
Let \(R_0,R_d\in\SO3\) denote the initial and desired attitudes of the satellite, respectively.
We construct a rest-to-rest maneuver with zero axial momentum.
The maneuver from \(R_0\) to \(R_d\) is decomposed into two steps.
First, the symmetry axis is rotated from \(\Gamma_0=R_0e_3\) to \(\Gamma_d=R_de_3\in\Sph^2\).
Second, the remaining rotation about the symmetry axis is generated by a closed loop of \(\Gamma(t)\) on \(\Sph^2\).

\paragraph{Pointing Alignment}

Let \([0,T_1]\) denote the time interval for the first step.
Let  \(\psi_1\in[0,\pi]\) and \(u_1\in\Sph^2\) be the rotation angle and rotation axis that carry \(\Gamma_0\) to \(\Gamma_d\).
When \(\Gamma_0\times\Gamma_d\neq 0\), they are given by
\begin{equation}
    \psi_1 = \cos^{-1}(\Gamma_0^T\Gamma_d),
    \qquad
    u_1 = \frac{\Gamma_0\times\Gamma_d}{\|\Gamma_0\times\Gamma_d\|}.
\end{equation}
If $\Gamma_0=\Gamma_d$, take $\psi_1=0$ and choose any unit $u_1\perp\Gamma_d$; then $R_1=R_0$, and this choice of $u_1$ is available if phase two is needed. If $\Gamma_0=-\Gamma_d$, take $\psi_1=\pi$ and choose any unit $u_1\perp\Gamma_0$, which is also perpendicular to $\Gamma_d$.
Choose a smooth function \(\psi:[0,T_1]\to[0,\psi_1]\) satisfying \(\psi(0)=0\) and \(\psi(T_1)=\psi_1\).
Require $\dot\psi(0)=\dot\psi(T_1)=0$ so this phase starts and ends at rest. To obtain continuous zero endpoint torques, also impose $\ddot\psi(0)=\ddot\psi(T_1)=0$.
The corresponding attitude trajectory may be written as
\begin{equation}
    R(t)=\exp(\psi(t)u_1^\wedge)R_0
        = R_0\exp(\psi(t)(R_0^Tu_1)^\wedge),
\end{equation}
for $t\in[0,T_1]$.
Its body-fixed angular velocity is \(\omega_b(t)=\dot\psi(t)R_0^Tu_1\).
Substitution into Euler's equation \eqref{eqn:dot_omega_b} gives the required torque,
{
\begin{equation}
    \tau(t) = J_1\ddot\psi(t)R_0^Tu_1.
\end{equation}
}%
Since \(R_0^Tu_1\) is orthogonal to \(e_3\), this torque has no component along the unactuated axis.
By construction, \(R(T_1)e_3=\Gamma_d\), so the first step aligns the symmetry axis with that of the desired attitude.

\paragraph{Axial Rotation via Geometric Phase}

After completion of the first step, the remaining attitude error is a rotation about the axis of symmetry.
Let \(R_1=R(T_1)\in\SO3\).
There exists \(\phi_2\in\mathbb{R}/2\pi\mathbb{Z}\) such that \(R_1^TR_d=\exp(\phi_2 e_3^\wedge)\).
 Choose a real representative $\phi_2\in[-\pi,\pi]$ by
\begin{equation}
    \phi_2 = \operatorname{atan2} \big( e_2^TR_1^TR_d e_1,\, e_1^TR_1^TR_d e_1 \big).
\end{equation}
The goal of the second phase is to generate this remaining axial rotation as a geometric phase.
If $\phi_2=0$, omit the second phase. Otherwise $\kappa$ below is either $+1$ or $-1$.
Let \(\kappa=\operatorname{sgn}(\phi_2)\), and choose the angular radius  \(\rho\in[0,\pi/3]\) so that \(2\pi(1-\cos\rho)=|\phi_2|\).
Define the center of the circular loop by
\begin{equation}
    u_2 = \cos\rho\,\Gamma_d + \kappa\sin\rho\,u_1.
\end{equation}
The corresponding circle on \(\Sph^2\) is
\begin{equation}
    \mathcal C_2 = \left\{ \Gamma\in\Sph^2 \;:\; u_2^T\Gamma=\cos\rho \right\}.
\end{equation}
Let \(s=t-T_1\in[0,T_2]\), and choose a smooth function \(\alpha:[0,T_2]\to[0,2\pi]\) satisfying \(\alpha(0)=0\) and \(\alpha(T_2)=2\pi\).
Require $\dot\alpha(0)=\dot\alpha(T_2)=0$, ensuring zero angular velocity at the splice and the end of the maneuver. If continuous zero endpoint torques are required, impose $\ddot\alpha(0)=\ddot\alpha(T_2)=0$ as well.
The desired reduced trajectory during the second phase is
\begin{equation}
    \Gamma(t) = \exp(\kappa\alpha(s)u_2^\wedge)\Gamma_d.
\end{equation}
This curve starts and ends at \(\Gamma_d\), and it encloses signed spherical area \(\phi_2\).
The attitude trajectory realizing this reduced motion is obtained by the horizontal lift
\begin{equation}
    R(t) = \exp(\kappa\alpha(s)u_2^\wedge) R_1 \exp(-\kappa\alpha(s)\cos\rho\,e_3^\wedge).
\end{equation}
 Using $\alpha(T_2)=2\pi$ and $\exp(2\pi\kappa u_2^\wedge)=I$, the terminal attitude is \[\begin{aligned}R(T_1+T_2)&=R_1\exp(-2\pi\kappa\cos\rho\,e_3^\wedge)\\&=R_1\exp(\phi_2e_3^\wedge)=R_d,\end{aligned}\] since $\phi_2=2\pi\kappa(1-\cos\rho)$.
It remains to verify that this trajectory can be generated by an admissible torque.
Differentiating the horizontal lift gives the body-fixed angular velocity \(\omega_b(t)=\kappa\dot\alpha(s)\eta(t)\), where
\begin{equation}
	\eta(t) = \exp\!\big(\kappa\alpha(s)\cos\rho\,e_3^\wedge\big) \big( R_1^Tu_2-\cos\rho\,e_3 \big).
\end{equation}
Since \(e_3^T R_1^Tu_2=\Gamma_d^Tu_2=\cos\rho\), we have \(e_3^T\eta(t)=0\), and hence \(e_3^T\omega_b(t)=0\).
 Thus the horizontal lift remains on the zero-axial-momentum level.
Finally, substituting \(\omega_b(t)=\kappa\dot\alpha(s)\eta(t)\) into Euler's equation gives the required torque for the second phase,
\begin{equation}
	\tau(t) = \kappa J_1\ddot\alpha(s)\eta(t) + J_1\dot\alpha(s)^2\cos\rho\, e_3\times\eta(t).
\end{equation}
This torque also satisfies \(e_3^T\tau(t)=0\), so the constructed maneuver is attainable by the underactuated satellite.

\subsection{Numerical Simulation}

{
We illustrate the two-stage satellite steering maneuver and assess how accurately finite-order Magnus synopsis variables approximate its geometric phase.
The spacecraft has inertia tensor $J=\operatorname{diag}(2.0,2.0,1.5)\,\mathrm{kg\,m^2}$ and starts at rest, so its conserved axial momentum is $\mu=0$.
The initial and desired attitudes are specified by the angle triples $(20^\circ,-35^\circ,50^\circ)$ and $(-45^\circ,60^\circ,-30^\circ)$, respectively.
Each triple is ordered as (roll, pitch, yaw), with the 3-2-1 convention $R=R_z(\mathrm{yaw})R_y(\mathrm{pitch})R_x(\mathrm{roll})$.

Both phases use the quintic profile $h(r)=10r^3-15r^4+6r^5$, whose first and second derivatives vanish at $r=0$ and $r=1$:
\begin{enumerate}
    \item \textbf{Phase 1} ($t\in[0,T_1]$, $T_1=5.0\,\mathrm{s}$): The symmetry axis moves from $\Gamma_0=R_0e_3$ to $\Gamma_d=R_de_3$ along the shortest geodesic, with $\psi(t)=\psi_1h(t/T_1)$ and $\psi_1=1.5670\,\mathrm{rad}$.
    \item \textbf{Phase 2} ($t\in[T_1,T_1+T_2]$, $T_2=10.0\,\mathrm{s}$): The symmetry axis traverses the circular loop $\mathcal C_2$ of angular radius $\rho=0.7127\,\mathrm{rad}$, starting and ending at $\Gamma_d$.
    With $s=t-T_1$, the loop is parameterized by $\alpha(s)=2\pi h(s/T_2)$ and traversed in the direction $\kappa=-1$.
    Its signed spherical area equals the residual axial attitude error, $\phi_2=-1.5293\,\mathrm{rad}$.
\end{enumerate}
These profiles ensure zero angular velocity and acceleration at the endpoints of each phase.
The attitude-error angle in Figure~\ref{fig:states_and_torques} is the geodesic error $\phi_{\mathrm{err}}(t)=\cos^{-1}((\operatorname{tr}(R_d^TR(t))-1)/2)\in[0,\pi]$, displayed in degrees.

As shown in \Cref{fig:states_and_torques}, Phase~1 eliminates the pointing mismatch, leaving an axial attitude error of $87.62^\circ$, corresponding to $|\phi_2|$.
The geometric phase generated in Phase~2 removes this residual error, bringing the spacecraft to the desired attitude at $T_1+T_2=15\,\mathrm{s}$.
Throughout the maneuver, $e_3^T\omega_b=0$ and $\tau_3=0$, consistent with zero axial momentum and actuation by transverse torques alone.

Figure~\ref{fig:convergence_rate} compares the terminal phase errors for Magnus truncations of the Phase~2 loop, with all retained coefficients initialized to zero at $t=T_1$.
The reference phase is the principal representative $\theta_{\mathrm{geo}}=\phi_2=-1.5293\,\mathrm{rad}$.
Only even orders are shown because the odd-order terms have zero axial component, as established in \chg{Section~\ref{sec:magnus-geophase}}.
The second-order approximation $b_b\cdot e_3$ has an absolute error of $0.6191\,\mathrm{rad}$ ($35.47^\circ$).
The error decreases at each plotted order $k=2,4,6,8,10$, reaching $0.0034\,\mathrm{rad}$ ($0.19^\circ$) at order ten.
 A least-squares fit of the natural logarithm $\ln E_k$ of the error against $k$ gives an empirical slope of approximately $-0.645$ per order.

For this loop,
\[
\int_{T_1}^{T_1+T_2}\|\omega_b(t)\|\,dt
=2\pi\sin\rho\approx4.108>\pi,
\]
so the sufficient Magnus convergence condition stated in \chg{Section~\ref{sec:magnus-geophase}} is not satisfied.
The observed error decay is therefore an empirical result over the computed orders.
Although $|\phi_2|<\pi$ identifies the principal reference phase, it does not by itself establish convergence of the Magnus series to the principal logarithm.

This numerical example shows that higher-order Magnus synopsis variables provide progressively more accurate approximations of the loop's geometric phase over the computed orders.
Since the stored endpoint already determines this phase modulo $2\pi$, the example illustrates the physical interpretation of the Magnus coordinates as an order-by-order description of the accumulated phase.
\par}

\begin{figure}[htbp]
    \centering
    \safeincludegraphics[width=0.85\textwidth]{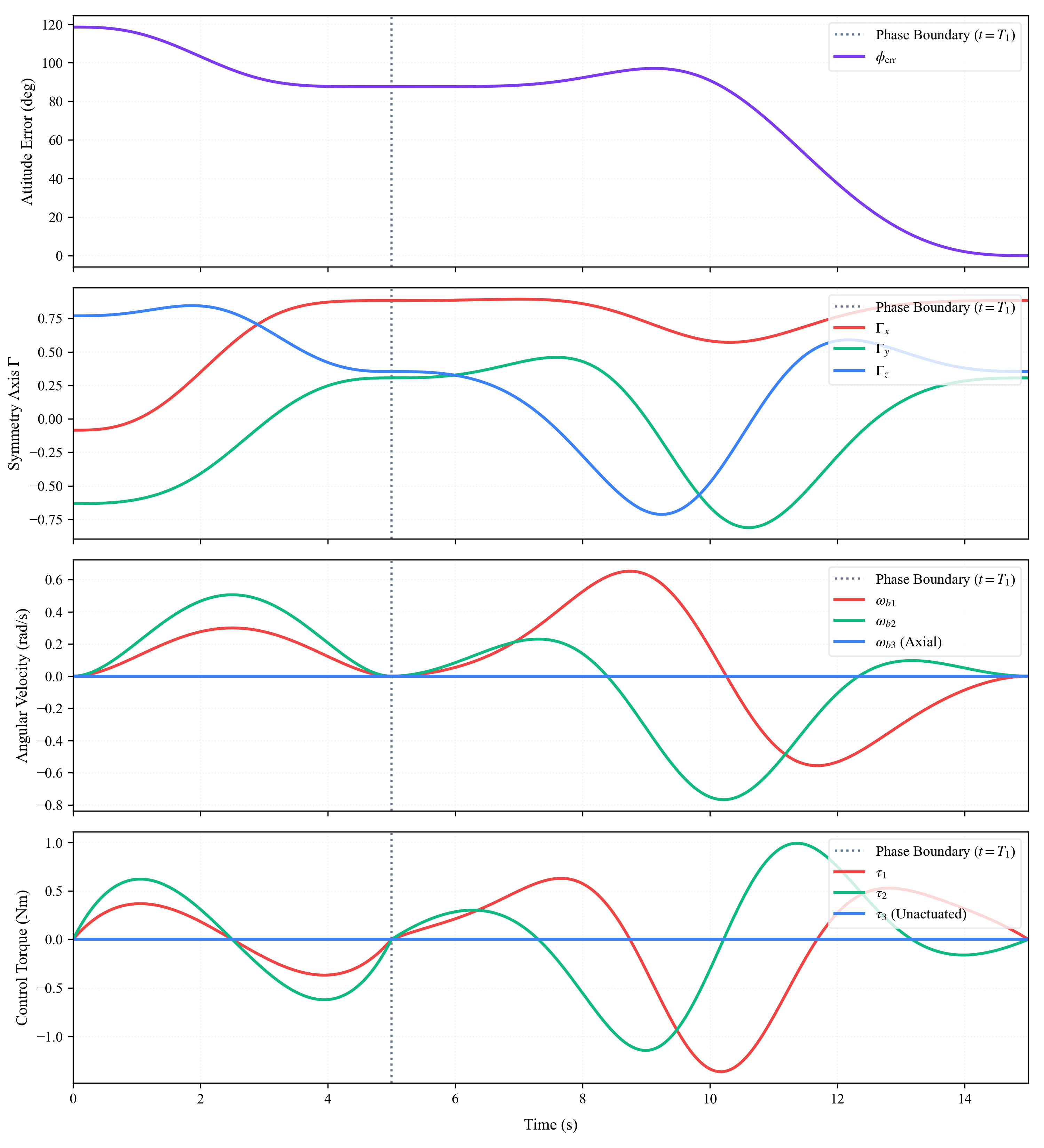}
    \caption{Time history of the underactuated satellite maneuver: attitude error $\phi_{\mathrm{err}}$, symmetry axis components $\Gamma$, body angular velocity $\omega_b$, and control torque $\tau$ across Phase 1 and Phase 2 ($T_1=5\,\text{s}$, $T_2=10\,\text{s}$).}
    \label{fig:states_and_torques}
\end{figure}

\begin{figure}[htbp]
    \centering
    \begin{tikzpicture}
        \node[anchor=south west,inner sep=0] (phaseplot) at (0,0)
            {\safeincludegraphics[width=0.65\textwidth]{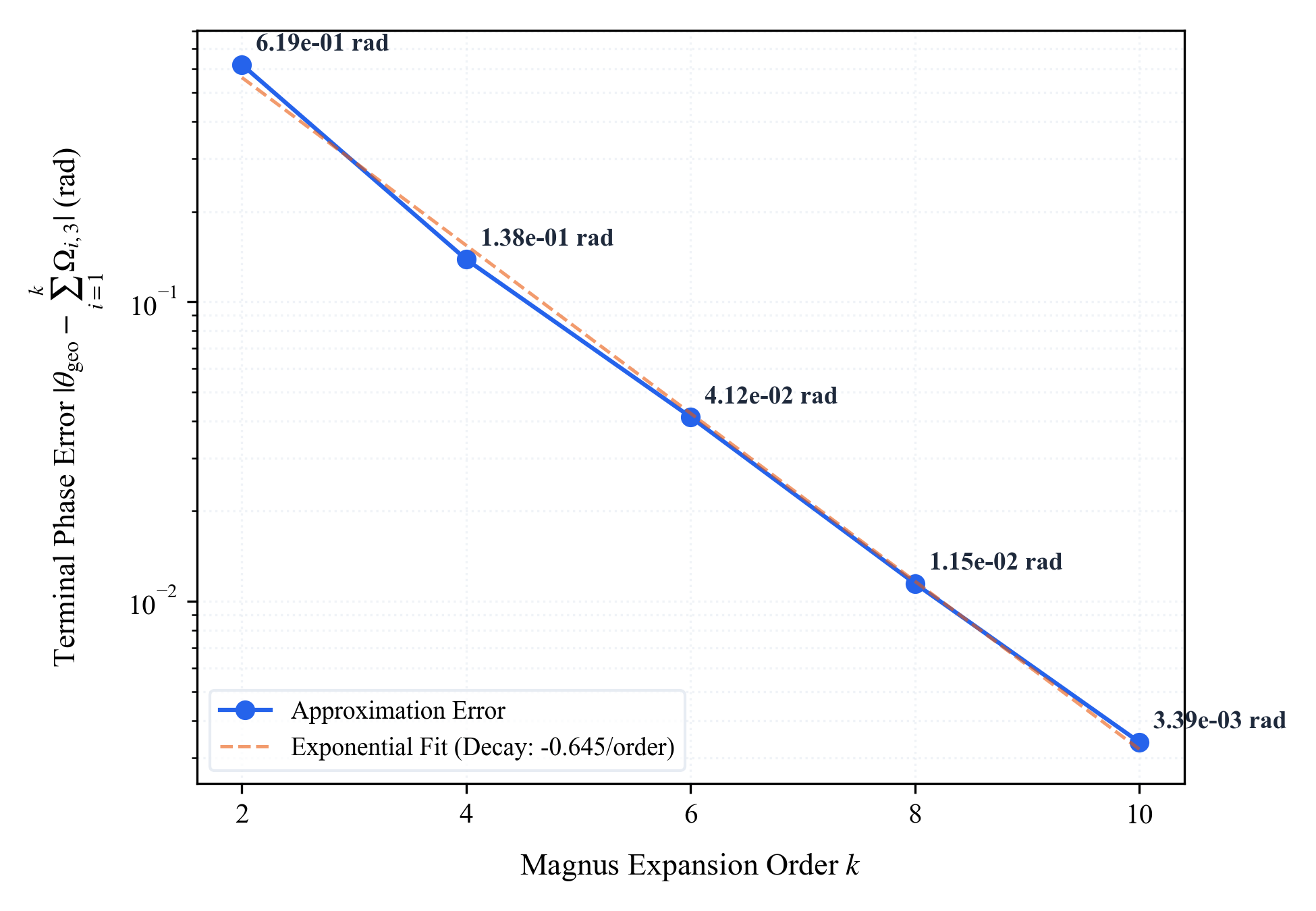}};
        \begin{scope}[x={(phaseplot.south east)},y={(phaseplot.north west)}]
            \fill[white] (0.012,0.14) rectangle (0.087,0.91);
            \node[rotate=90,inner sep=0] at (0.049,0.553)
                {\fontsize{7}{8}\selectfont Phase error (rad)};
        \end{scope}
    \end{tikzpicture}
    \caption{Terminal phase error $E_k=\left|\theta_{\mathrm{geo}}-\sum_{i=1}^{k}e_3^T\Xi_{i,b}^{\vee}\right|$ (rad) for the Phase~2 loop, as a function of the Magnus expansion order $k$.}
    \label{fig:convergence_rate}
\end{figure}

\section{Conclusion}

The synopsis framework presented here, which consists of retained endpoints coupled with truncated Chen or Magnus signatures, gives composable finite-dimensional descriptions of rotational motion in body, spatial, and combined coordinates. The construction preserves the exact terminal attitude while retaining selected integral information about how it was reached. The all-orders results establish the group laws, inverses, and finite matrix realizations, and show precisely why endpoint retention adds information beyond any fixed finite Chen truncation.

The full Chen hierarchy and the evaluated Magnus hierarchy offer different storage and information choices: they agree at order two but diverge from order three onward. Conditional finite-modal reconstruction and the geometric-phase example connect these descriptors to interpretable features of rotational motion. The classical signature and Lie-group results provide the algebraic foundations; the explicit rotational constructions provide states that can be composed and implemented. Assessing their benefit in estimation and filtering is a concrete direction for subsequent work.
{The geometric-phase example illustrates interpretability rather than necessity, since for a closed loop with $\mu=0$  the stored endpoint determines the phase modulo $2\pi$.  Proposition~\ref{prop:endpoint-adds} resolves realizability for the body, spatial, and hybrid Chen synopsis groups, and the quotient maps give the same result for the evaluated Magnus groups. Propagation of uncertainty through the composition laws remains an open question.}

\section*{Acknowledgements}
This work was aided by the use of  ChatGPT 5.2 (including the generation of Figure 1), and  ChatGPT 6 and Claude Sonnet 5 for editing and assistance with proofs, and was supported by internal funds from the University of Delaware and Mohamed bin Zayed University of Artificial Intelligence.

\end{document}